\documentclass[11pt]{imsart}
\usepackage[hiresbb]{graphicx}
\usepackage{fullpage}
\usepackage{graphicx}
\usepackage{amsmath}
\usepackage{amssymb}
\usepackage[round]{natbib}
\usepackage{color}
\usepackage{url}
\usepackage{amsthm}
\usepackage{latexsym}

\newtheorem{thm}{Theorem}[section]
\newtheorem{cor}{Corollary}[section]
\newtheorem{lem}{Lemma}[section]
\newtheorem{rem}{Remark}[section]
\newtheorem{example}{Example}[section]
\newtheorem{prop}{Proposition}[section]

\providecommand{\keywords}[1]{\textbf{\textit{Keywords:}} #1}

\newcommand{\revise}[1]{\textcolor{black}{#1}}

\DeclareMathOperator*{\argmax}{arg\,max}

\newcommand{\mGamma}[2]{
\raise1.5pt\hbox{$\displaystyle \mathop{\Gamma}^{\rm m}$}\hspace{-1.5pt}_{#1}^{\, #2}}
\newcommand{\eGamma}[2]{
\raise1.5pt\hbox{$\displaystyle \mathop{\Gamma}^{\rm e}$}
\hspace{-1.5pt}_{#1}^{\, #2}}
\newcommand{\aGamma}[2]{
\raise1.5pt\hbox{$\displaystyle \mathop{\Gamma}^{\alpha}$}\hspace{-1.5pt}_{#1}^{\, #2}}
\newcommand{\minusaGamma}[2]{
\raise1.5pt\hbox{$\displaystyle \mathop{\Gamma}^{-\alpha}$}\hspace{-1.5pt}_{#1}^{\, #2}}
\newcommand{\aaGamma}[2]{
\raise1.5pt\hbox{$\displaystyle \mathop{\Gamma}^{\alpha_{0}}$}\hspace{-1.5pt}_{#1}^{\, #2}}
\newcommand{\zeroGamma}[2]{
\raise1.5pt\hbox{$\displaystyle \mathop{\Gamma}^{0}$}\hspace{-1.5pt}_{#1}^{\, #2}}
\newcommand{\halfGamma}[2]{
\raise1.5pt\hbox{$\displaystyle \mathop{\Gamma}^{1/2}$}\hspace{-1.5pt}_{#1}^{\, #2}}
\newcommand{\malpha}[2]{
\raise1.5pt\hbox{$\displaystyle \mathop{\alpha}^{\rm m}$}\hspace{-1.5pt}_{#1}^{\, #2}}
\newcommand{\mH}[2]{
\raise1.5pt\hbox{$\displaystyle \mathop{H}^{\rm m}$}\hspace{-1.5pt}_{#1}^{\, #2}}
\newcommand{\mnabla}[2]{
\raise1.5pt\hbox{$\displaystyle \mathop{\nabla}^{\rm m}$}
\hspace{-1.5pt}_{#1}^{\, #2}}
\newcommand{\ealpha}[2]{
\raise1.5pt\hbox{$\displaystyle \mathop{\alpha}^{\rm e}$}
\hspace{-1.5pt}_{#1}^{\, #2}}
\newcommand{\enabla}[2]{
\raise1.5pt\hbox{$\displaystyle \mathop{\nabla}^{\rm e}$}\hspace{-1.5pt}_{#1}^{\, #2}}

\begin{document}
\begin{frontmatter}
\begin{aug}
\title{Matching prior pairs \\ connecting Maximum A Posteriori estimation \\ and posterior expectation}
\runtitle{Matching prior pairs}
\author[A]{\fnms{Michiko} \snm{Okudo}\ead[label=e1]{okudo@mist.i.u-tokyo.ac.jp}}
\and
\author[B]{\fnms{Keisuke} \snm{Yano}\ead[label=e2]{yano@ism.ac.jp}}
\address[A]{Department of Mathematical Informatics,
Graduate School of Information Science and Technology,
The University of Tokyo
7-3-1 Hongo, Bunkyo-ku, Tokyo, 113-8656, JAPAN.
\printead{e1}
}
\address[B]{The Institute of Statistical Mathematics,  10-3 Midori cho, Tachikawa City, Tokyo, 190-8562, Japan. 
\printead{e2}}
\end{aug}

\begin{abstract}
Bayesian statistics has
two common measures of central tendency of a posterior distribution: posterior means and Maximum A Posteriori (MAP) estimates.
In this paper, we discuss a connection between MAP estimates and posterior means. 
We derive an asymptotic condition for a pair of prior densities under which the posterior mean based on one prior coincides with the MAP estimate based on the other prior.
A sufficient condition for the existence of this prior pair relates to $\alpha$-flatness of the statistical model in information geometry. We also construct a matching prior pair using $\alpha$-parallel priors.
Our result elucidates an interesting connection between regularization in generalized linear regression models and posterior expectation.
\end{abstract}

\keywords{Bayesian inference,
Generalized linear regression, Information geometry, Prior selection}

\end{frontmatter}

%%%%%%%%%%%%%%%%%%%%%%%%%%%%%%%%%%%%%%%%%%%%%%%%%%%%%%%%%%%%%%

%%%%%%%%%%%%%%%%%%%%%%%%%
\section{Introduction}
%%%%%%%%%%%%%%%%%%%%%%%%%
\label{sec:intro}

In Bayesian statistics,
two common measures of central tendency of a posterior distribution are posterior mean and Maximum A Posteriori (MAP) estimate.
Posterior mean is the Bayes estimate, an estimate minimizing the expected loss for a squared-error loss function.
This is usually computed by the expectation using Markov chain Monte Carlo (MCMC).
MAP estimate lies in the literature of penalized likelihood estimate or regularized maximum likelihood estimate. This is obtained by the optimization.
Although the computational schemes of two estimates are different, the celebrated Bernstein--von-Mises theorem 
tells that 
in the first-order asymptotic regime of the sample size,
the posterior shape becomes Gaussian with the center equal to the MAP estimate. Thus the posterior mean and the MAP estimate become the same in the asymptotic regime.
Yet, 
practical behaviours of these estimates are quite different (e.g., \citealp{pananos2020}). 
Recent studies (\citealp{gribonval2011,gribonval2013,louchet2013,burger2014}) highlight differences and connections between these estimates in several statistical models. In particular, 
\cite{gribonval2013} reveals that in Gaussian linear inverse problems, 
although the posterior mean and the MAP estimate for the same prior may be different,
every posterior mean based on a prior is also the MAP estimate based on a different prior.

To elucidate a further connection between MAP estimates and posterior means in general statistical models,
this paper derives the asymptotic condition for a pair of priors $(\pi,\widetilde{\pi})$
under which the posterior mean derived from one prior $\pi$ coincides with the MAP estimate based on the other prior $\widetilde{\pi}$.
We call this pair of priors matching prior pair.
From our discovery of matching prior pairs,
we see that
in a generalized linear regression model,
although
the posterior mean based on a Gaussian prior
may be different from the ridge regression (the MAP estimate based on a Gaussian prior), 
the matching prior pair of the Gaussian prior can deliver 
the MAP estimate closer to the posterior mean based on the Gaussian prior in asymptotic regimes.
\revise{Although our main discovery is theoretical, it can also have practical implications, in particular, in the computation of each estimate.}
When there exists a difficulty in optimizing the log posterior density to obtain a MAP estimate for a prior,
we can utilize the posterior mean based on another prior that forms a matching pair with the given prior.
In contrast,
when it is hard to build an MCMC for computing a posterior mean, we can instead evaluate a MAP estimate that matches to the posterior mean.
\revise{This point will be further clarified by numerical experiments in Section \ref{sec:experiments}.}

The existence of a matching prior pair has an information-geometrical flavor \citep{amari1985}.
The information geometry presents a class of $\alpha$-connections $(\alpha\in\mathbb{R})$ concerning the manifold of probability distributions.
We show that a matching prior pair exists 
for an $\alpha$-affine parameterization, that is, the parameterization with $\alpha$-connection equal to $0$. Further, we also provide an explicit construction of the matching prior pair using $\alpha$-parallel priors (\citealp{takeuchi2005}).
This information-geometrical notation appears because the posterior expectation elicits the information about the flatness of the statistical model with respect to the $(-1)$-connection as observed in \cite{komaki1996} and in \cite{okudo2021}.

There is a literature on bridging the gap between the MAP estimation and the posterior expectation.
In the objective Bayesian literature, a prior yielding the posterior mean asymptotically equal to the maximum likelihood estimate (MLE) is \revise{called the moment matching prior}.
\cite{ghosh2011} derives a formula for constructing a moment matching prior.
\cite{hashimoto2019} extends the construction to non-regular statistical models. 
\cite{yanagimoto2023} extends the moment matching prior to the conditional inference.
Our matching prior pair includes the moment matching prior and naturally extends its idea to the MAP estimate based on a non-uniform distribution.
\cite{gribonval2013} reveals an elegant construction of an exact matching prior pair for linear inverse problems.
\cite{polson2016} proposes an exact prior pair that matches a density with a MAP estimate plugged-in and a marginal density.
These results are exact in the sense that it holds even in the finite regime of sample size but are limited to several models.
Although our construction relies on the asymptotics with respect to the sample size, it elucidates the connection in general statistical models using information geometry.

The rest of this paper is structured as follows.
Section \ref{sec:marchingpriorpairs} delivers the main result, an information-geometrical construction of a matching prior pair.
Section \ref{sec:examples} displays analytical examples that examine the main result.
Section \ref{sec:experiments} presents numerical examples using synthetic and real data.
All technical proofs are presented in Section \ref{sec:proofs}.

\if0
%曲指数モデルを考えている
We have observations
$y^n = \{y(1),y(2),\dots,y(n)\}~~(y(1),\dots,y(n)\in\mathbb{R}^p)$
 independently distributed according to a probability distribution with a density function $p(y;\theta)$ that belongs to a statistical model
%
\begin{align*}
\mathcal{P}
=\left\{p(y;\theta) \mid \theta \in \theta \right\}.
\end{align*}
%
We consider a statistical model of a curved exponential family
%
\begin{align*}
\mathcal{P}
= &\{p(y;\theta)=s(y)
\exp(\theta^i(\theta) y_i-\Psi(\theta(\theta)))  \mid 
\theta=(\theta^a)\in \theta, a=1,\dots,d, \ i=1,\dots,p \},
\end{align*}
%
where $\theta \subset \mathbb{R}^d$ and $1 \leq d \leq p$.
We consider estimation of $\theta$ based on Bayesian methods.
%While Bayes estimators are intractable in most cases, a posterior mean estimator and a MAP estimator are often considered.
\fi

%%%%%%%%%%%%%%%%%%%%%%%%%%%%%%%%%%%%%%%%%%%%%%%%%%%%%
\section{Matching prior pairs}
%%%%%%%%%%%%%%%%%%%%%%%%%%%%%%%%%%%%%%%%%%%%%%%%%%%%%
\label{sec:marchingpriorpairs}

We first prepare several notations for the theory and then present the construction of the matching prior pair.

\subsection{Preparation}

Let $\mathcal{Y}$ be a sample space and let $dy$ be a base measure.
Assume that we have observations
$y^n = \{y(1),y(2),\dots,y(n)\}~~(y(1),\dots,y(n)\in\mathcal{Y})$
 independently distributed according to a probability distribution with a density function $p(y\,;\,\theta)$ that belongs to a statistical model
 parameterized by $\theta$:
\begin{align*}
\mathcal{P}
=\left\{p(y\,;\,\theta) \mid \theta \in \Theta \right\}
\,\,\text{with}\,\,\Theta\subset \mathbb{R}^{d}.
\end{align*}
We denote by $\mathrm{E}_{\theta}$ the expectation with respect to the density with $\theta$.

For the theory, we first introduce several information-geometric notations; for details, see \cite{amari1985}.
Components of the Fisher information matrix $g=(g_{ab})_{a,b=1,\ldots,d}$ are defined as
\[
g_{ab}(\theta)
:=
\int 
p(y\,;\,\theta)\{
\partial_{a}\log p(y\,;\,\theta)\} 
\{\partial_{b}\log p(y\,;\,\theta)\}
dy
=
\int \frac{\partial_a p(y\,;\,\theta)\partial_b p(y\,;\,\theta)}{p(y\,;\,\theta)}dy
\]
where $\partial_a = \partial/\partial \theta^a$.
For $a,b=1,\ldots,d$, 
let $g^{ab}$ be a component of the inverse matrix of the Fisher information matrix $g$.
For $a,b,c=1,\ldots,d$,
the m-connection coefficient (the $(-1)$-connection coefficient) and e-connection coefficient (the $1$-connection coefficient) are defined as
\begin{equation}
\begin{split}
\mGamma{abc}{} 
&:= \int \frac{\partial_{a}\partial_{b} p(y\,;\,\theta) \partial_{c} p(y\,;\,\theta)}{p(y\,;\,\theta)}dy \mbox{~~and~~}\\
\eGamma{abc}{} 
&:= \int p(y;\theta) \{\partial_{a}\partial_{b}\log p(y\,;\,\theta) \}\{\partial_{c}\log p(y\,;\,\theta)\} dy,
\label{re: e-m-connections}
\end{split}
\end{equation}
respectively.
For $a,b,c=1,\ldots,d$,
let
\begin{align*}
 T_{abc}
:=\mGamma{abc}{}-\eGamma{abc}{}
=\int p(y;\theta)\{\partial_a\log p(y\,;\,\theta)\}\{\partial_b\log p(y\,;\,\theta)\}\{\partial_c\log p(y\,;\,\theta)\}
dy.
\end{align*}
Further,
the $\alpha$-connection coefficient for $\alpha\in\mathbb{R}$ is defined as
\begin{align}
\aGamma{abc}{} := \mGamma{abc}{} - \frac{1+\alpha}{2}T_{abc}
=\eGamma{abc}{}+\frac{1-\alpha}{2}T_{abc}
\,\,\,(a,b,c=1,\ldots,d).
\label{eq: alpha connection}
\end{align}
These connections form dual connections, that is,
\begin{align}
\partial_{a}g_{bc}(\theta)
=\aGamma{abc}{}(\theta)
+\minusaGamma{acb}{}(\theta)
\quad
\text{for any}\,\,\alpha\in\mathbb{R}. 
\label{eq: dual structure}
\end{align}
To ease the notation, we use the Einstein summation convention: if an index occurs twice in any one term, once as an upper and once as a lower index, summation over that index is implied.
Let 
\[
T_{a}:=T_{abc}g^{bc}
\quad \text{and}\quad
\aGamma{ab}{~c} := \aGamma{abe}{}g^{ce}
\,\,\,\,\text{for}\,\, a,b,c=1,\ldots,d.\]

We then prepare notions of flatness in the information geometry.
For given $\alpha\in\mathbb{R}$,
the statistical model $\mathcal{P}$ is called $\alpha$-flat if and only if there exists a parameterization with $\aGamma{abc}{}=0$ for all $a,b,c=1,\ldots,d$; e.g., p.47 of \cite{amari1985}.
The parameterization with $\aGamma{abc}{}=0$ for $a,b,c=1,\ldots,d$ is called an $\alpha$-affine coordinate.
Further, the statistical model is said to be  statistically equi-affine when
$\partial_{a}T_{b}=\partial_{b}T_{a}$ for $\theta\in\Theta$; e.g., Definition 2 of \cite{takeuchi2005}.
The concept of statistical equi-affinity is important to the existence of the subsequent $\alpha$-parallel priors and we have a handy sufficient condition for the statistically equi-affinity as described below.
\begin{lem}[\cite{lauritzen1987}; Propositions 3 and 4 of \cite{takeuchi2005}]
\label{lem:statisticallyequiaffine}
If the model is $\alpha$-flat for a certain $\alpha\ne0$, it is statistically equi-affine.
\end{lem}
We thirdly introduce $\alpha$-parallel priors.
For $\alpha\in\mathbb{R}$, an 
$\alpha$-parallel prior $\pi_{\alpha}$ proposed by \cite{takeuchi2005} is defined by 
\begin{align}
\partial_a \log \pi_\alpha (\theta) = \aGamma{ab}{~b}(\theta) \,\,\,(a =1,\ldots,d)
\label{eq: alpha parallel prior}
\end{align}
if it exists.
This class contains several non-informative priors proposed in the objective Bayesian literature (e.g., \citealp{tanaka2023}).
First, it includes the well-known Jeffreys prior $\pi_{\mathrm{J}}(\theta):=|g(\theta)|^{1/2}$ with the determinant $|\cdot|$ 
as the $0$-parallel prior:
\[
\partial_{a}\log \pi_{\rm J}(\theta) = \frac{1}{2}\partial_{a}\log |g|(\theta)=\frac{1}{2}g^{bc}(\theta)\partial_{a}g_{bc}(\theta)
=\zeroGamma{ab}{~b}(\theta)
\,\,\,(a=1,\ldots,d).
\]
Second, it has the $\chi^{2}$-prior $\pi_{\chi^{2}}(\theta)$ proposed by  \cite{liu2014} as $1/2$-parallel prior:
\[
\partial_{a}\log \pi_{\chi^{2}}(\theta)=\halfGamma{ab}{~b}(\theta)
\,\,\,(a=1,\ldots,d),
\]
which is recently pointed out by \cite{tanaka2023}.
For $\alpha=1$, we call this e-parallel prior $\pi_{\rm e}$,
and for $\alpha=-1$, we call this m-parallel prior $\pi_{\rm m}$.
\cite{takeuchi2005} find the following lemmas for the existence of $\alpha$-parallel priors.
\begin{lem}
[Proposition 2 of \cite{takeuchi2005}]
A statistically equi-affine statistical model has $\alpha$-parallel priors for any $\alpha\in\mathbb{R}$. Otherwise, the model has only the $0$-parallel prior.
\end{lem}
\begin{lem}
[Proposition 5 of \cite{takeuchi2005}]
\label{lem: sufficient condition for alpha parallel prior}
If a statistical model is $\alpha_0$-flat for a certain $\alpha_0\neq 0$,
then the model has $\alpha$-parallel priors for arbitrary $\alpha\in\mathbb{R}$.
\end{lem}
\begin{example}[Exponential family]
Consider an exponential family with
a sufficient statistic $T(y)\in\mathbb{R}^{d}$:
\[
\{\, 
\exp\{\theta^{\top}T(y)-\psi(\theta)\}m(y) \mid 
\theta\in\Theta
\},
\]
where $\psi(\theta)$ is the potential function and $m(y)$ is a given function.
The exponential family is $\mathrm{e}$- \& $\mathrm{m}$-flat ($\pm 1$-flat), that is, it has parameterizations $\theta$ and $\eta=\mathrm{E}_{\theta}[T(Y)]$ with
$\eGamma{ab}{~c}(\theta)=0$ and $\mGamma{ab}{~c}(\eta)=0$, respectively.
Thus, Lemma \ref{lem: sufficient condition for alpha parallel prior} implies that exponential families have $\alpha$-parallel priors for arbitrary $\alpha$. In particular, the $\mathrm{e}$-parallel prior is a uniform prior with respect to $\theta$, and
the $\mathrm{m}$-parallel prior is a uniform prior with respect to $\eta$.
\end{example}

We conclude this subsection by introducing a condition on the existence of the solution of a certain partial differential equation.
For a differentiable function $h(\theta)=(h_{1}(\theta),\ldots, h_{d}(\theta))^{\top}$, consider the partial differential equation of $u(\theta)$ given by
\[
\partial_{a}u (\theta)=h_{a}(\theta)\quad (a=1,\ldots,d).
\]
\begin{lem}[p.~348 of \cite{handbookde}; Section 1-3 of \cite{matsudaiwanami}]
\label{lem:compatible}
If $\partial_{b}h_{a}(\theta)=\partial_{a}h_{b}(\theta)$ for $1\le a < b \le d$,
then $u(\theta)$ exists locally. 
\end{lem}

%%%%%%%%%%%%%%%%%%%%%%%%%%%%%%%%%%%%%%%%%%%%%%%%%%%%%
\subsection{Main results}
%%%%%%%%%%%%%%%%%%%%%%%%%%%%%%%%%%%%%%%%%%%%%%%%%%%%%
\label{sec:expansion}
We present information-geometrical condition and construction of matching prior pairs.
We begin with an asymptotic condition of matching prior pairs for general models, and then we derive a simple form for an $\alpha$-affine coordinate in a statistically equi-affine model.
Lastly, we derive matching prior pairs for other statistics including variance and higher-order moments.
In the rest of the paper,
we assume regularity conditions 
\revise{for asymptotic expansions of posterior mean estimators and MAP estimators, and we consider only the class of priors that satisfy the regularity conditions.
%to ensure that the maximum likelihood estimate and the Bayes estimate behave well asymptotically, and posterior tail probabilities are negligible.
For the details of the regularity conditions, see Supplementary Material. 
%\cite{hartigan1998}.
}

The following theorem delivers an asymptotic condition to match a posterior mean based on a prior $\pi_{\mathrm{PM}}$ and a MAP estimate based on a prior $\pi_{\mathrm{MAP}}$
except for terms of ${\rm o}_p(n^{-1})$.
The proof is given in Section \ref{sec:proofs}.
\begin{thm}
\label{thm_match}
%%%
The posterior mean $\hat{\theta}_{\rm PM}$ based on a prior $\pi_{\rm PM}$ and the MAP estimate $\hat{\theta}_{\rm MAP}$ based on a prior $\pi_{\rm MAP}$ coincide except for terms of $o_p(n^{-1})$ when the prior pair $(\pi_{\mathrm{PM}},\pi_{\mathrm{MAP}})$ satisfies,
for $a=1,\ldots,d$,
\begin{align}
\label{eq:match}
\partial_{a}\log\frac{\pi_{\rm PM}}{\pi_{\rm MAP}}(\hat{\theta}_{\rm MLE})
-
\left(
\partial_{a}\log \pi_{\rm J}(\hat{\theta}_{\rm MLE})
+\frac{1}{2}g^{cd}(\hat{\theta}_{\rm MLE})
\eGamma{cda}{}(\hat{\theta}_{\rm MLE})
\right)=\mathrm{o}_p(n^{-1}),
\end{align}
where $\hat{\theta}_{\rm MLE}$ is the MLE.
\end{thm}

\begin{rem}
Theorem \ref{thm_match} includes the construction of the moment matching prior
proposed by \cite{ghosh2011}.
The moment matching prior is the prior that yields the posterior mean asymptotically equal to the MLE.
\cite{tanaka2023} rewrites the partial differential equation for the moment matching prior $\pi_{\rm MM}$ in an information-geometrical way:
\[
\partial_{a}\log \pi_{\rm MM} (\theta)
-\left(
\partial_{a}\log{\pi_{\rm J}}(\theta)
+ \frac{1}{2}g^{cd}(\theta)\eGamma{cda}{}(\theta)
\right)=0
\,\,\,(a=1,\ldots,d).
\]
As the prior that yields MLE as the MAP estimate is a uniform prior,
the moment matching prior $\pi_{\rm MM}$ satisfies (\ref{eq:match}) as the matching prior pair of a uniform prior.
\end{rem}
\revise{
\begin{rem}
From the objective Bayesian perspective,
the usage of the MAP estimate is somewhat controversial as the MAP estimate is not invariant with respect to the parameterization; see 
\cite{druilhetmarin2007} for the literature.
To resolve this issue,
\cite{druilhetmarin2007} proposes JMAP estimate, that is, the MAP estimate obtained from the original prior $\pi$ divided by the Jeffreys prior $\pi_{\rm{J}}$.
Our result also tells the connection between the posterior mean based on a prior $\pi_{\rm{PM}}$ and the JMAP estimate based on a prior $\pi_{\rm{JMAP}}$.
Putting $\pi_{\rm{JMAP}}/\pi_{\rm{J}}$ to $\pi_{\rm{MAP}}$, 
equation (\ref{eq:match}) becomes
\[
\partial_{a} \log \frac{\pi_{\rm{PM}}}{\pi_{\rm{JMAP}}}(\hat{\theta}_{\rm{MLE}})
-\frac{1}{2}g^{cd}(\hat{\theta}_{\rm{MLE}})
\eGamma{cda}{}(\hat{\theta}_{\rm{MLE}})
=o_{p}(n^{-1}),
\]
which implies seeking a prior pair between the posterior mean and the JMAP estimate is more directly related to the behavior of $\mathrm{e}$-connection coefficients.
\end{rem}
}

\begin{rem}
\revise{
In connection with the invariance in the previous remark,
we should mention the lack of invariance in a matching prior pair.
In fact,
equation (\ref{eq:match}) is not invariant with respect to parameterization, which implies that an explicit form of a matching prior pair depends on parameterization.
Consider changing the parameterization $\theta$ to 
$\xi$.
By using change of variables,
equation (\ref{eq:match}) for the new parameterization $\xi$ becomes  
\begin{align*}
&\frac{\partial \theta^{a}}{\partial \xi^{a'}}(\hat{\theta}_{\rm{MLE}})
\left\{
\frac{\partial}{\partial \theta^{a}}\log \frac{\pi_{\rm{PM}}}{\pi_{\rm{MAP}}}(\hat{\theta}_{\rm{MLE}})
-\left(\frac{\partial}{\partial \theta^{a}}\log \pi_{\rm{J}}(\hat{\theta}_{\rm{MLE}})
+\frac{1}{2}
g^{cd}(\hat{\theta}_{\rm{MLE}})
\eGamma{cda}{}(\hat{\theta}_{\rm{MLE}})
\right)
\right.\\
&\left.  \quad \quad\quad \quad \quad 
-\frac{1}{2}
g^{cd}(\hat{\theta}_{\rm{MLE}})
g_{ab}(\hat{\theta}_{\rm{MLE}})
\frac{\partial \xi^{c'}}{\partial \theta^{c}}(\hat{\theta}_{\rm{MLE}})
\frac{\partial \xi^{d'}}{\partial \theta^{d}}(\hat{\theta}_{\rm{MLE}})
\frac{\partial^{2}\theta^{b}}{
\partial \xi^{c'} \partial \xi^{d'} 
}(\hat{\theta}_{\rm{MLE}})
\right\}
=o_{p}(n^{-1}).
\end{align*}
So, even the existence of a matching prior pair depends on parameterization.
This is mainly because the connection coefficient is not a tensor,
and 
 is reasonable because the form of the posterior mean itself changes according to the parameterization.
 Given this fact, we shall discuss the existence and the construction of a matching prior pair 
 below.
 }
\end{rem}

In a one-dimensional statistical model,
an explicit construction of a matching prior pair is easy.
Consider the following ordinary differential equation:
\begin{align*}
\frac{d}{d\theta}\log\frac{\pi_{\rm PM}(\theta)}{\pi_{\rm MAP}(\theta)\pi_{\rm J}(\theta)}=\frac{1}{2}g^{11}(\theta)\eGamma{111}{}(\theta).
\end{align*}
The integration with respect to $\theta$ yields, for arbitrary $\theta_{0}\in\Theta$,
\begin{align*}
\frac{\pi_{\rm PM}(\theta)}{\pi_{\rm MAP}(\theta)}
\propto \pi_{\rm J}(\theta) \exp\left\{\int_{\theta_{0}}^{\theta} \left(
\frac{1}{2}g^{11}(\theta')\eGamma{111}{}(\theta')d\theta'
\right)
\right\}.
\end{align*}

Yet, in a multi-dimensional statistical model, 
even the existence of a matching prior pair is non-trivial.
We then seek a sufficient condition for the existence of a matching prior pair and an explicit construction of the pair.
The following corollary provides a sufficient condition of the existence and a explicit construction using information geometry.
\begin{cor}
Assume that the model is statistically equi-affine
and $\theta$ is $\alpha$-affine for a certain $\alpha\in\mathbb{R}$.
Then, a matching prior pair exists.
Further, the prior pair $(\pi_{\rm PM},\pi_{\rm MAP})$ satisfying
\begin{align}
\label{eq:matchingprior}
\frac{\pi_{\rm PM}(\theta)}{\pi_{\rm MAP}(\theta)}
\propto \pi_{\rm J}(\theta)\left(\frac{\pi_{\rm e}(\theta)}{\pi_{\rm m}(\theta)}\right)^{(1-\alpha)/4}
\end{align}
is a matching prior pair; that is,
 the posterior mean based on $\pi_{\rm PM}$ and the MAP estimate based on $\pi_{\rm MAP}$ coincide except for terms of $o_p(n^{-1})$.
\end{cor}
\begin{proof}
Observe that we have $\eGamma{abc}{}(\theta)=-\{(1-\alpha)/2\}T_{abc}(\theta)$ ($a,b,c=1,\ldots,d$) for an $\alpha$-affine coordinate $\theta$.
Together with (\ref{eq: alpha connection}), this implies that for an $\alpha$-affine coordinate $\theta$,
the condition (\ref{eq:match}) becomes
\begin{align*}
\partial_{a}\left(
\log \frac{\pi_{\rm PM}(\theta)}{\pi_{\rm MAP}(\theta)\pi_{\rm J}(\theta)}
\right)=-\frac{1-\alpha}{4}T_{a}(\theta)
+{\rm o}_{p}(n^{-1})
\quad
\text{for $\theta=\hat{\theta}_{\rm MLE}$}.
\end{align*}
Then, consider the following partial differential equation:
\begin{align}
\label{eq:match2}
\partial_{a}\left(
\log \frac{\pi_{\rm PM}(\theta)}{\pi_{\rm MAP}(\theta)\pi_{\rm J}(\theta)}
\right)=-\frac{1-\alpha}{4}T_{a}(\theta).
\end{align}
Lemma \ref{lem:compatible} tells that this equation has a solution if
\begin{align*}
\partial_{b}T_{a}(\theta)=\partial_{a}T_{b}(\theta) \quad (1\le a < b \le d),
\end{align*}
which is equal to the statistical equi-affinity of the model
and thus a matching prior pair exists
for an $\alpha$-affine coordinate $\theta$ in a statistically equi-affine model.

Further, by the definition of the e-\&m-parallel priors (\ref{eq: alpha parallel prior}), we have
\[
\partial_{a}\log \frac{\pi_{\rm m}(\theta)}{\pi_{\rm e}(\theta)}=T_{a}(\theta) \,\,\,(a=1,\ldots,d),
\]
and obtain
\[
\partial_{a}\left(
\log \frac{\pi_{\rm PM}(\theta)}{\pi_{\rm MAP}(\theta)\pi_{\rm J}(\theta)}
\right)=
\partial_{a}
\log \left(\frac{\pi_{\rm e}(\theta)}{\pi_{\rm m}(\theta)}
\right)^{(1-\alpha)/4}
\quad(a=1,\ldots,d).
\]
So, the prior pair (\ref{eq:matchingprior}) is a matching prior pair, which completes the proof.
\end{proof}

\revise{
From Lemma \ref{lem:statisticallyequiaffine}, 
$\alpha$-affine coordinates satisfy the assumption above. In the following subsections, we shall give several such examples in submodels of exponential families. However, readers may consider practical applications beyond submodels of exponential families. Although finding the matching prior pair may be difficult in general statistical models, in another direction, one-step calibration between the MAP estimate and the posterior expectation based on a prior is possible using the following corollary.
}
\revise{
\begin{cor}
\label{cor_calibration}
For a prior $\pi$,
the posterior expectation $\hat{\theta}_{\rm{PM}}$ is calculated by
\begin{equation}
\hat{\theta}_{\rm{PM}}^{a} =
\hat{\theta}_{\rm{MAP}}^{a} + \frac{1}{2n}g^{ab}(\hat{\theta}_{\rm{MAP}})g^{cd}(\hat{\theta}_{\rm{MAP}})
\left\{\frac{1}{n}
\sum_{t=1}^{n}\partial_{bcd}\log p(y(t)\,;\,\hat{\theta}_{\rm{MAP}})
\right\}
+o_{p}(n^{-1})
\end{equation}
for $a=1,\ldots,d$.
\end{cor}
This calibration formula is derived from several ingredients of the proof of the main theorem. 
It can also be obtained through the asymptotic expansion in \cite{miyata2004} (see also \citealp{yanagimoto2023}); however, it has not been used for our specific purpose, that is, calibrating the posterior expectation and the MAP. In practice, the higher-order derivatives in the formula can be efficiently computed using automatic differentiation (cf. \citealp{Iri1984, baydinetal2018}). The supplementary material checks the validity of this calibration.
}

We conclude this section with the following extension of  matching prior pairs, that is, matching prior pairs for other statistics including higher-order moments. 
\begin{prop}
\label{prop_moment_general}
Let $f(\theta)=(f_{1}(\theta),\ldots,f_{d}(\theta))$ be a third-times differentiable function.
If two priors $\pi_{\rm PM}$ and $\pi_{\rm MAP}$ satisfy
\begin{align*}
\partial_a  f_{i} (\theta) \left\{\partial_b\log\frac{\pi_{\rm PM}}{\pi_{\rm MAP}}(\theta)-
\partial_{b}\log \pi_{\rm J}(\theta)
-\frac{1}{2}g^{cd}(\theta)
\eGamma{cdb}{}(\theta)
\right\}
-\frac{1}{2}\partial_a\partial_b f_{i}(\theta)
= \mathrm{o}_p(n^{-1})
\end{align*}
at $\theta=\hat{\theta}_{\rm MLE}$
for $i=1,\ldots,d$ and $a=1,\ldots,d$,
the posterior expectation $f_{\pi_{\rm PM}}$ of $f$ 
based on $\pi_{\rm PM}$
and the MAP-plugged-in estimate $f(\hat{\theta}_{\rm MAP})$ 
based on $\pi_{\rm MAP}$
coincide except for $\mathrm{o}_p(n^{-1})$-terms.
\end{prop}

%%%%%%%%%%%%%%%%%%%%%%%%%%%%%%
\subsection{Examples}
\label{sec:examples}
%%%%%%%%%%%%%%%%%%%%%%%%%%%%%%
In this section, 
we present matching prior pairs (\ref{eq:matchingprior}) in a submodel of an exponential family:
\[
\{
\exp\{\theta^{\top}(\xi)T(y)-\psi(\theta(\xi))\}m(y) \mid \xi\in\Xi
\},
\]
where $\theta$ is the $p$-dimensional canonical/natural parameter of the exponential family,
and $\xi$ is a $d$-dimensional model parameter.
This includes the exponential family itself and the generalized linear regression model with a canonical link function.
As the exponential family is e-\&m-flat and has $\alpha$-parallel priors,
their e- or m-flat submodels also have 
$\alpha$-parallel priors
and so we confine ourselves to e- or m-flat submodels of an exponential family.

%%%
\subsubsection{Generalized linear models with canonical links and regression coefficients}
%%%
%If the model belongs to an exponential family and the model parameter is $\theta$, the e-connection coefficient satisfies $\eGamma{abc}{}=0$.
We first consider a generalized linear regression model (GLM) with a canonical link function:
\[
\{
\exp(\theta^{\top}(\beta)T(y)-\psi(\theta(\beta)))m(y) \mid \beta\in\Xi
\}
\,\,\,\,\text{with}\,\,
\theta=X\beta,
\]
where
$\beta$ is an unknown regression coefficient of $d$ dimension ($p\ge d$),
and
$X$ is a given full rank matrix $X\in\mathbb{R}^{p\times d}$.
In this case,
since
$\theta$ is the e-affine coordinate and
$\partial \theta/\partial \beta=X^{\top}$,
the e-connection coefficient with respect to $\beta$ also vanishes
$\eGamma{abc}{}(\beta)=0$,
it suffices to seek the prior pair satisfying
\begin{align}
\partial_a\log\frac{\pi_{\rm PM}}{\pi_{\rm MAP}}(\beta)
-
\left(
\partial_{a}\log \pi_{\rm J}(\beta)
\right)
=0,
\end{align}
which is equal to
\begin{align}
\label{eq:e-flat}
\frac{\pi_{\rm PM}(\beta)}{\pi_{\rm MAP}(\beta)}
\propto \pi_{\rm J}(\beta).
\end{align}

As a simple example,
consider a Gaussian model with mean zero and unknown variance: the data $y^n=\{y(1),\dots,y(n)\}$ independently come from a Gaussian distribution $N(0,\sigma^2)$.
Here we employ an inverse-gamma prior \[
\sigma^2\sim{\rm InvGamma}(a,b)
\quad\text{with}\,\, a,b>0.
\]
Consider the canonical parameter $\theta=\sigma^{-2}$
and the posterior mean of $\theta$.
In this parameterization,
the prior becomes 
\[\theta\sim{\rm Gamma}(a,b).
\]
Let $\pi_{\rm PM}$ denote its density.
The posterior distribution of $\theta$ is ${\rm Gamma}(a+n/2, b+\sum\nolimits_{i=1}^n y(i)^2/2)$,
and then
the posterior mean based on $\pi_{\rm PM}$ is
\[
\hat{\theta}_{\rm PM} = \frac{a+n/2}{b+\sum\nolimits_{i=1}^n y(i)^2/2}.
\]
By using (\ref{eq:e-flat}), we set $\pi_{\rm MAP}(\theta) \propto \pi_{\rm PM}(\theta)/\pi_{\rm J}(\theta)$, where $\pi_{\rm J}(\theta) \propto \theta^{-1}$.
Since the log posterior density based on $\pi_{\rm MAP}$ is 
\[
\log p(y^n\,;\,\theta) + \log\pi_{\rm MAP}(\theta)
= \left(a + n/2 -1 \right) \log\theta - (b+\sum\nolimits_{i=1}^n y(i)^2/2)\theta + \log\theta + C
\]
with the constant $C$ independent from $\theta$,
 the MAP estimate based on $\pi_{\rm MAP}$ is
\[
\hat{\theta}_{\rm MAP} = \frac{a+n/2}{b+\sum\nolimits_{i=1}^n y(i)^2/2},
\]
which implies the exact matching 
$\hat{\theta}_{\rm PM} = \hat{\theta}_{\rm MAP}$.

%%%
\subsubsection{m-flat submodels and m-affine parameters}
%%%
We proceed to a linear submodel with respect to the expectation parameter $\eta=\mathrm{E}_{\theta}[T(Y)]$ of the exponential family:
\[
\{
\exp(\theta^{\top}(\xi)T(y)-\psi(\theta(\xi)))m(y) \mid \xi\in\Xi
\}
\,\,\,\,\text{with}\,\,
\eta=\eta(\theta)=X\xi,
\]
where
$\eta\in\mathbb{R}^p$,
$\xi$ is a model parameter of $d$ dimension $(p\ge d)$,
and
$X$ is a given full rank matrix $X\in\mathbb{R}^{p\times d}$.
In this case, the m-connection coefficients $\mGamma{abc}{}(\xi)$ ($a,b,c=1,\ldots,d$) vanish and 
\[
\partial_{a}\log \pi_{\rm J}(\xi)=\zeroGamma{ab}{b}(\xi)=\left(\mGamma{ab}{b}(\xi)-\frac{1}{2}T_{a}(\xi)\right)=-\frac{1}{2}T_{a}(\xi)
=-\frac{1}{2}\partial_{a}\log \frac{\pi_{\rm m}(\xi)}{\pi_{\rm e}(\xi)},
\]
which implies 
\[
\frac{\pi_{\rm e}(\xi)}{\pi_{\rm m}(\xi)}\propto \left\{\pi_{\rm J}(\xi)\right\}^{2}.
\]
Then, the condition (\ref{eq:matchingprior}) becomes
\begin{align}
\label{eq:m-flat}
\frac{\pi_{\rm PM}(\xi)}{\pi_{\rm MAP}(\xi)} \propto \pi_{\rm J}(\xi)\left\{
(\pi_{\rm J}(\xi))^{2}\right\}^{1/2}
=\left\{\pi_{\rm J}(\xi)\right\}^{2}.
\end{align}

As a simple example, consider a Poisson model:
the data $y^n=\{y(1),\dots,y(n)\}$ independently come from a Poisson distribution ${\rm Poisson}(\lambda)$. 
Let us consider a Gamma prior $\lambda\sim{\rm Gamma}(a,b)$ with $a,b>0$ for $\lambda$ and denote the density by $\pi_{\rm PM}$.
Then the posterior distribution is the Gamma distribution ${\rm Gamma}(a+\sum\nolimits_{i=1}^n y(n), b+n)$,
and 
the posterior mean based on $\pi_{\rm PM}$ is
\[
\hat{\lambda}_{\rm PM} = \frac{a+\sum\nolimits_{i=1}^n y(n)}{b+n}.
\]
Since $\lambda$ is the expectation parameter and $\pi_{\rm J}(\lambda)\propto \lambda^{-1/2}$, 
the matching prior pair satisfies
\[
\frac{\pi_{\rm PM}(\lambda)}{\pi_{\rm MAP}(\lambda)} \propto \lambda^{-1}.
\]
We set $\pi_{\rm MAP}(\lambda) = \lambda \pi_{\rm PM}(\lambda)$.
Since 
the log posterior density based on $\pi_{\rm MAP}$ is
\[
\log p(y^n\,;\,\lambda) + \log\pi_{\rm MAP}(\lambda)
= \left(a + \sum_{i=1}^n y(i)-1 \right) \log\lambda - (b + n)\lambda + \log\lambda + C
\]
with the constant $C$ independent from $\lambda$,
 the MAP estimate based on $\pi_{\rm MAP}$ is
\[
\hat{\lambda}_{\rm MAP} = \frac{a+\sum\nolimits_{i=1}^n y(n)}{b+n},
\]
which implies the exact matching
$\hat{\lambda}_{\rm PM} = \hat{\lambda}_{\rm MAP}$.

%%%%%%%%%%%%%%%%%%%%%%%%%%%%%%
\section{Numerical experiments}
\label{sec:experiments}
%%%%%%%%%%%%%%%%%%%%%%%%%%%%%%

In this section, we examine the theory using the Bayesian logistic regression model and the Poisson shrinkage model.

\subsection{The Bayesian Logistic regression}

\begin{figure}[h]
    \centering
    \includegraphics[width=110mm]{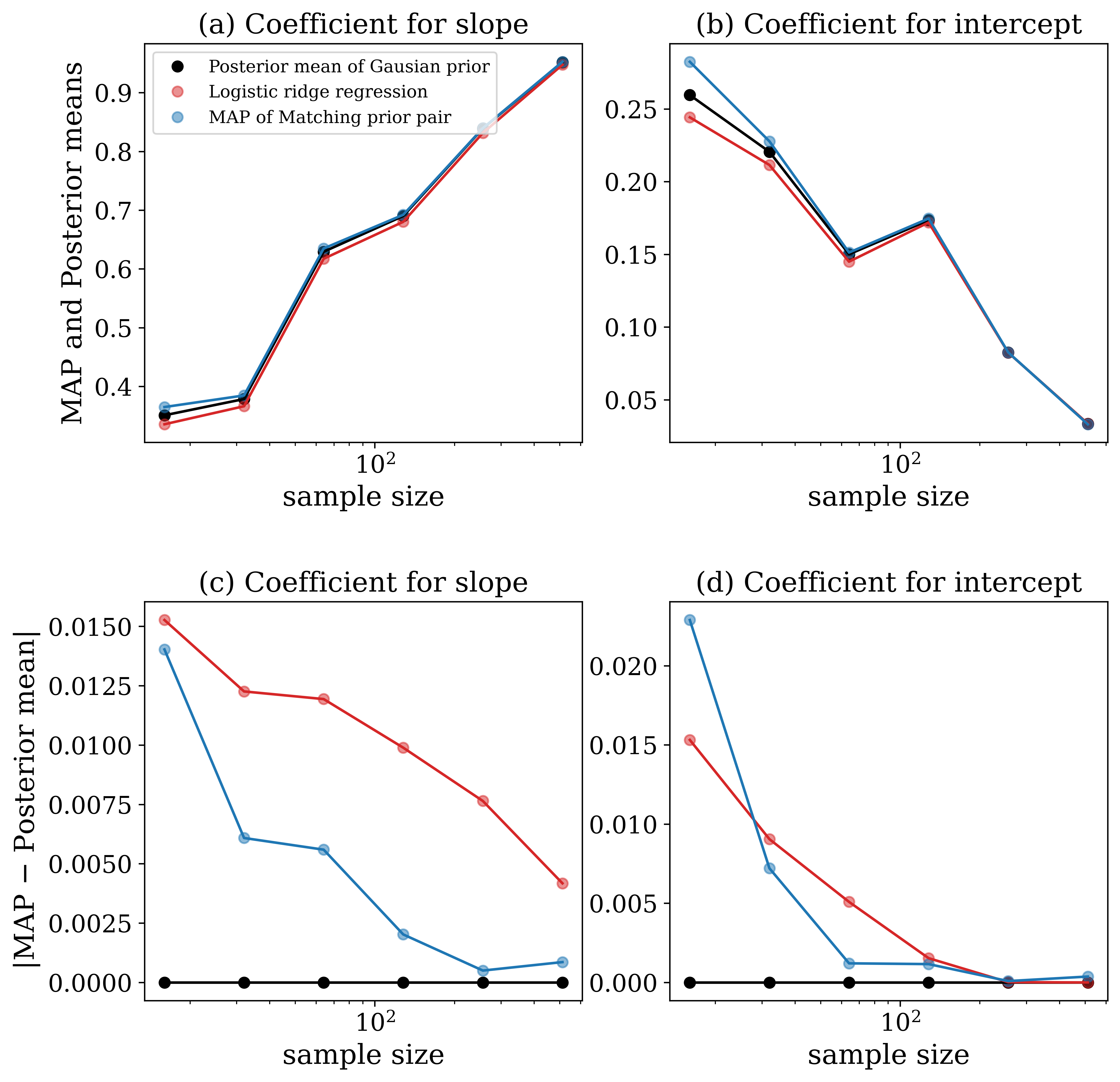}
    \caption{\revise{The posterior mean based on the Gaussian prior (colored in black), the logistic ridge regression (colored in red), and the MAP estimate based on the matching prior pair (colored in blue) under the correctly specified logistic regression model. (a)-(b) estimates themselves; (c)-(d) the differences with respect to the posterior mean based on the Gaussian prior. The horizontal axis adopts a logarithmic scale.}}
    \label{fig:Logisticregression_synthetic_scenario1}
\end{figure}

\begin{figure}[h]
    \centering
    \includegraphics[width=100mm]{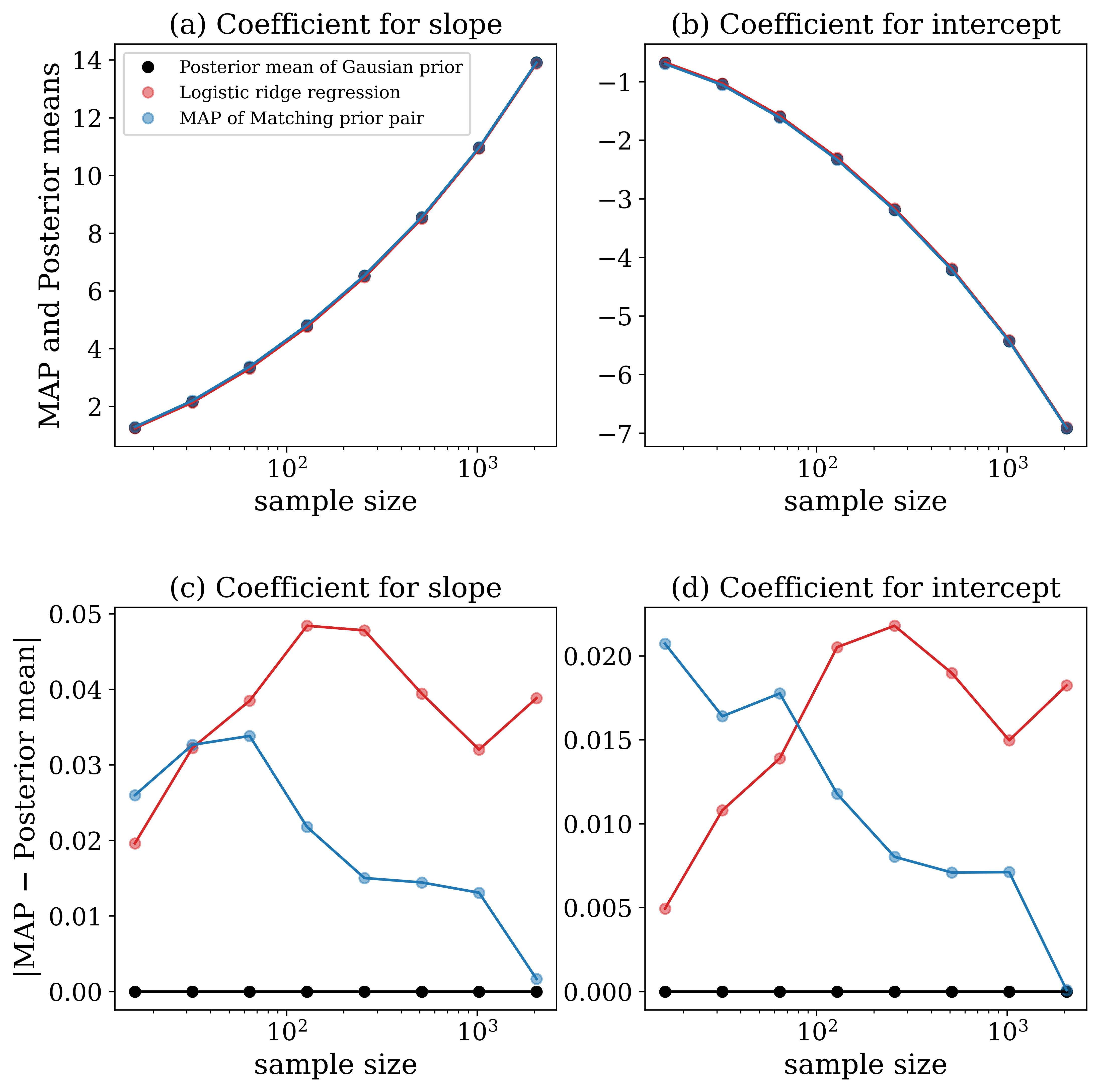}
    \caption{
    \revise{The posterior mean based on the Gaussian prior (colored in black), the logistic ridge regression (colored in red), and the MAP estimate based on the matching prior pair (colored in blue) under the misspecified logistic regression model. (a)-(b) estimates themselves; (c)-(d) the differences with respect to the posterior mean based on the Gaussian prior. The horizontal axis adopts a logarithmic scale.}}
    \label{fig:Logisticregression_synthetic_scenario2}
\end{figure}

\begin{figure}[h]
    \centering
    \includegraphics[width=110mm]{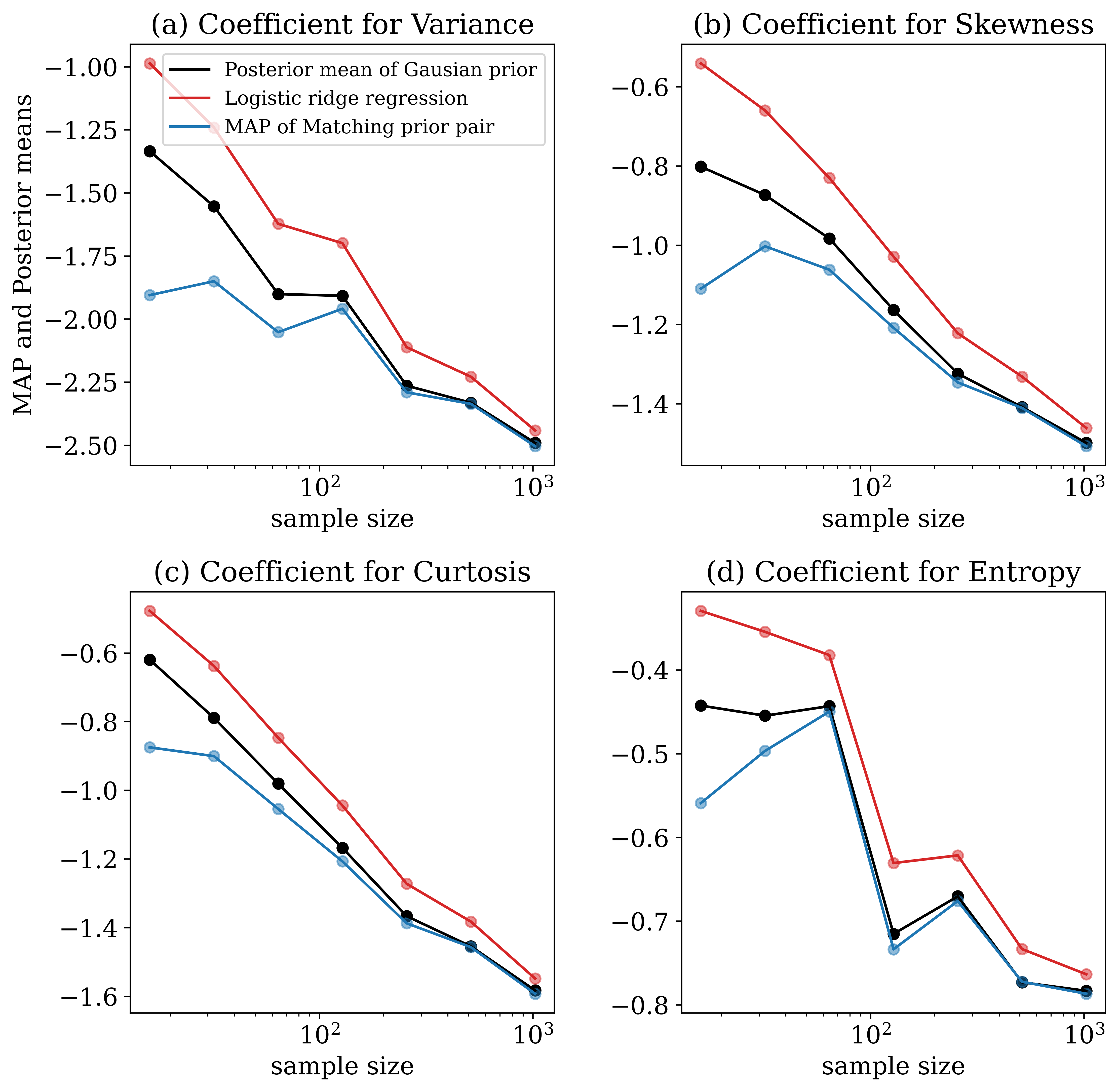}
    \caption{
    \revise{The Posterior mean based on the Gaussian prior (colored in black),  the logistic ridge regression (colored in red), and the MAP estimate based on the matching prior pair (colored in blue) for the banknote authentification data.}}%; (c)-(d) the differences with respect to the posterior mean based on the Gaussian prior. The horizontal axis adopts a logarithmic scale.}
\label{fig:Logisticregression_real_bank}
\end{figure}

Bayesian logistic regression model is a Bayesian version of the popular logistic regression model.
By putting a Gaussian prior on regression coefficients,
the working model becomes
\begin{align*}
Y(i) \mid X(i) \,,\, \beta &\sim 
\mathrm{Bernoulli}(\sigma(X(i)\beta) )\,\, (i=1,\ldots,n),
\\
\beta &\sim \mathrm{Normal}(0,I),
\end{align*}
where $\sigma(x)=1/\{1+\exp(-x)\}$.
This Bayesian model has been sometimes related to the logistic ridge regression:
\begin{align*}
\argmax_{\beta} \left\{\sum_{i=1}^{n} l(y(i),x(i)\,;\,\beta)  - \frac{\|\beta\|^{2}}{2}\right\}
\end{align*}
with $l(y(i),x(i)\,;\,\beta):=y(i)\log \sigma(x(i)\beta)+(1-y(i))\log \{1-\sigma(x(i)\beta)\}$.
Our theory tells a gap between the posterior mean based on the Gaussian prior $\pi_{\mathrm{PM}}(\beta)$ and
the logistic ridge regression (the MAP estimate based on $\pi_{\mathrm{PM}}$).
Also, the matching prior pair (\ref{eq:match}) gives another prior $\pi_{\mathrm{MAP}}$ yielding the MAP estimate asymptotically equal to the posterior mean of $\pi_{\mathrm{PM}}$
\begin{align*}
\pi_{\mathrm{MAP}}(\beta)\propto \left|\sum_{i=1}^{n}(X(i))^{\top}X(i)\left[\sigma(X(i)\beta)\{1-\sigma(X(i)\beta)\}\right]
\right|^{-1/2}\pi_{\mathrm{PM}}(\beta),
\end{align*}
inducing the following optimization:
\begin{align*}
\argmax_{\beta} \sum_{i=1}^{n} l(y(i),x(i);\beta)  - \frac{\|\beta\|^{2}}{2} - \frac{1}{2}\log \left| \sum_{i=1}^{n}(X(i))^{\top}X(i)\left[\sigma(X(i)\beta)\{1-\sigma(X(i)\beta)\}\right] \right|.
\end{align*}

First, we check the behaviour of the matching prior pair by using the following two synthetic data:\revise{
\begin{align*}
X^{(1)}(i)&=i/n,\\
Y^{(1)}(i)\mid X^{(1)}(i) &\sim
\mathrm{Bernoulli}(\sigma(X^{(1)}(i)+0.0)),
\end{align*}
}
and
\begin{align*}
(Y^{(2)}(i),X^{(2)}(i))=
\begin{cases}
(1,i/n) \,\,\text{if}\,\, i > n/2, \\
(0,i/n) \,\,\text{if otherwise}.
\end{cases}
\end{align*}
The logistic regression model with 2-dimensional parameters (slope and intercept) applied to the former data $\{(Y^{(1)}(i),X^{(1)}(i))\,:\,i=1,\ldots,n\}$ is correctly-specified,
while the model applied to the latter data $\{(Y^{(2)}(i),X^{(2)}(i))\,:\,i=1,\ldots,n\}$ is misspecified.
The former data is random and so we take the mean of the performance using 50 repetitions. For the calculation of the posterior mean, we use the 10000 Markov chain Monte Carlo samples after the 10000 burnin samples by conducting the P\'{o}lya-Gamma augmentation (\citealp{polson2016}).
We vary the sample size 
in $\{2^{t}: t=4,5,6,7,8,9\}$ for the former case
and 
in $\{2^{t}: t=4,5,6,7,8,9,10,11\}$
for the latter case, respectively.

Figures \ref{fig:Logisticregression_synthetic_scenario1} and \ref{fig:Logisticregression_synthetic_scenario2} display the results. From Figure \ref{fig:Logisticregression_synthetic_scenario1}, we see that under the correctly-specified model, the gap between the logistic ridge regression and the posterior mean based on the Gaussian prior is larger than the gap between the MAP estimate based on the matching prior pair and the posterior mean based on the Gaussian prior. Both gaps become smaller as the sample size gets larger. Figure \ref{fig:Logisticregression_synthetic_scenario2} showcases the performance under the misspecified model. The performance with small sample sizes seems random but with moderate or large sample sizes, the MAP estimate based on the matching prior pair gets closer to the posterior mean based on the Gaussian prior.
In both cases, there exists a gap between the logistic ridge regression and the posterior mean based on the Gaussian prior, and the matching prior pair reduces this gap.

\revise{
We further examine the computational time for obtaining 
the posterior mean and
the MAP estimates based on the matching prior pair in the synthetic dataset as in Figure \ref{fig:Logisticregression_synthetic_scenario1}. 
We calculate the mean and standard deviation using 10 repetitions. Table \ref{tab: computational time Logistic} indicates that
the optimization is relatively fast compared to the MCMC algorithm, particularly in regimes with large sample sizes. Therefore, the MAP estimate based on the matching prior pair can serve as a useful approximation of the posterior mean while providing fast computational times.
} 
\begin{table}[h]
\label{tab: computational time Logistic}
\caption{\revise{Computational time for obtaining 
the posterior mean and
the MAP estimates based on the matching prior pair. The set-up is the same as in Figure \ref{fig:Logisticregression_synthetic_scenario1}.}}
\begin{center}
\begin{tabular}{|r|c|c|}
\hline
  Sample size $n$ & Posterior mean & MAP of Matching prior pair \\\hline
  $2^4$ & 4.72 s ($\pm 0.07$ s) & 0.01 s ($\pm 0.00$ s) \\
  $2^5$ & 5.13 s ($\pm 0.05$ s) & 0.01 s ($\pm 0.00$ s) \\
  $2^6$ & 6.12 s ($\pm 0.07$ s)& 0.02 s ($\pm 0.00$ s)\\
  $2^7$ & 8.15 s ($\pm 0.08$ s) & 0.05 s ($\pm 0.00$ s)
  \\
  $2^8$ & 12.7 s ($\pm 0.07$ s)& 0.08 s ($\pm 0.01$ s)\\
  $2^9$ & 45.8 s ($\pm 2.72$ s) & 0.17 s ($\pm 0.00$ s)\\
  \hline
\end{tabular}
\end{center}
\end{table}

Next, we check the performance of the matching prior pair by the banknote authentication data from UCI Machine Learning Repository
(\citealp{UCI}).
The banknote authentication data set classifies genuine and forged banknote-like specimens based on four image features (Variance, Skewness, Curtosis, and Entropy).
The number of unknown parameters in this case is 4.
We check the performance for sample sizes of $\{2^{t}: t=4,5,6,7,8,9,10\}$.
For each sample size, we take indices randomly taken 50 times and then take the average of the performance.

Figure \ref{fig:Logisticregression_real_bank} displays the result. For all four variables, the matching prior pair reduces the gap with respect to the posterior mean based on the Gaussian prior except for small sample sizes. Although there seem some biases for the coefficients of Skewness and Curtosis, we note that there exist deviations from the theoretical values of the posterior means due to the randomness in MCMC.
Overall, the calibration based on the matching prior pair works well for the logistic regression model.

\subsection{The Poisson shrinkage model}
Poisson sequence model is a canonical model for count-data analysis. 
Recently, incorporating the high-dimensional structure with Poisson sequence model has been well investigated 
(\citealp{komaki2004,datta2016,yanokanekokomaki2021,hamuraetal2022}).

\begin{figure}[h]
    \centering
\includegraphics[width=110mm]{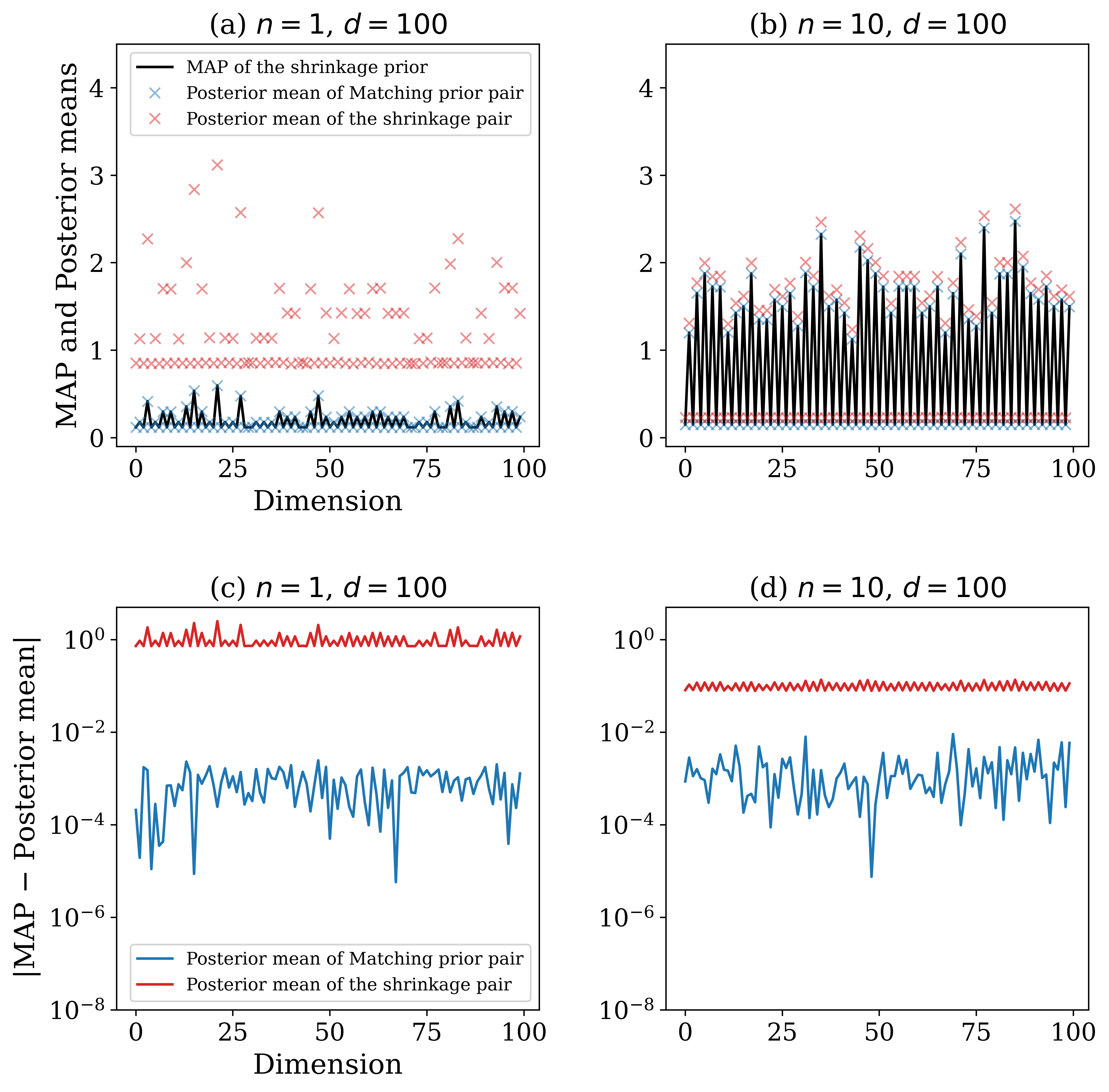}
    \caption{\revise{The MAP estimate based on the shrinkage prior (colored in black), the posterior mean based on the shrinkage prior (colored in red), and the posterior mean based on the matching prior pair (colored in blue) under the synthetic Poisson sequence model.}}%; (d)-(f) the differences with respect to the MAP estimate based on the shrinkage prior.}
    \label{fig:Poisson_syn}
\end{figure}
\begin{figure}[h]
    \centering
    \includegraphics[width=105mm]{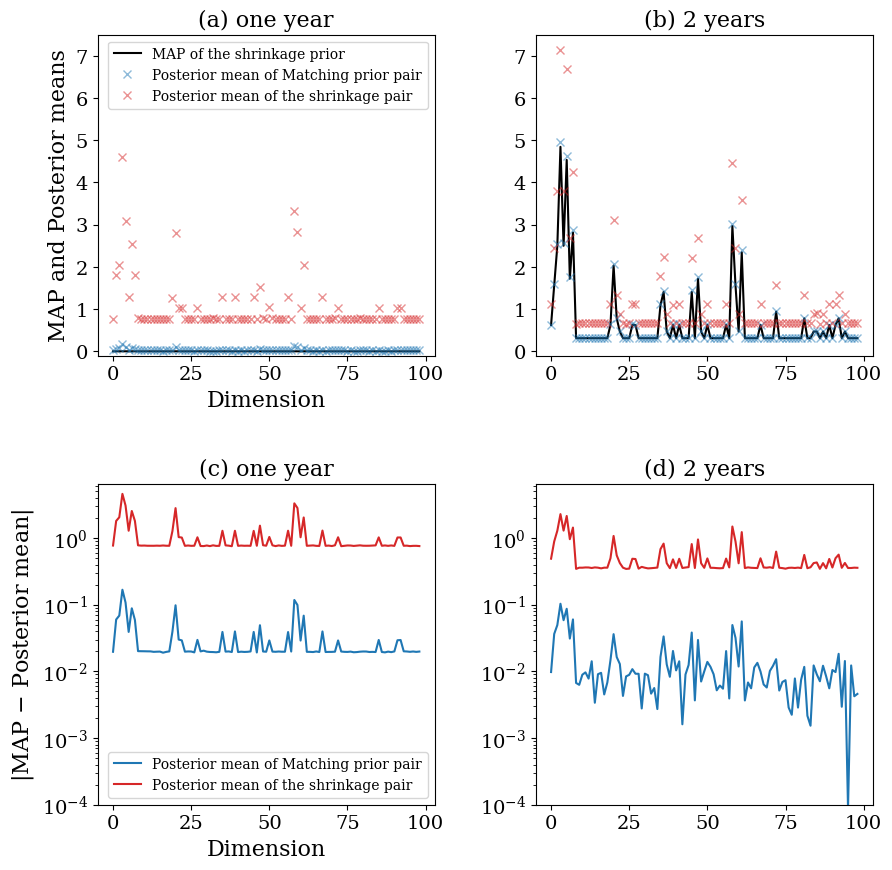}
    \caption{The MAP estimate based on the shrinkage prior (colored in black), the posterior mean based on the shrinkage prior (colored in red), and the posterior mean based on the matching prior pair (colored in blue) for pickpockets in Chuo ward, Tokyo, Japan. \revise{(a) estimates using only data for the entire year of 2012; (b) estimates using data from the two-year period of 2012-2013; (c)-(d) the differences with respect to the MAP estimate based on the shrinkage prior.}}
    \label{fig:Poisson_real_chuo}
\end{figure}

\revise{
The working model here is 
\[
Y(t)=\begin{pmatrix}
Y_{1}(t)\\
Y_{2}(t)\\
\vdots\\
Y_{d}(t)
\end{pmatrix} \quad \mid \quad \lambda =  \begin{pmatrix}
\lambda_{1}\\
\lambda_{2}\\
\vdots\\
\lambda_{d}
\end{pmatrix}
\quad\sim\quad
\otimes_{i=1}^{d} \mathrm{Poisson}(\lambda_{i})
\quad (t=1,\ldots,n),
\]
where $t$ is an index for the observation, $i$ is an index for the coordinate,
and $\otimes$ denotes the product of measures.
In this model,
each observation $Y(t)$
given $\lambda$
follows from the $d$-dimensional independent Poisson distribution.
In the application to spatio-temporal count-data analysis,
$t$ may be the index for the year and $i$ may be the index for the observation site such as a district; see \cite{datta2016,yanokanekokomaki2021,hamuraetal2022} for the details.
}

\revise{We investigate the calibration based on the matching prior pair in high-dimension and under an improper prior. We work with the improper shrinkage prior proposed by \cite{komaki2006}:
\begin{align*}
\pi(\lambda)&= \frac{\lambda_{1}^{\beta_{1}-1}\cdots\lambda_{d}^{\beta_{d}-1}}{(\lambda_{1}+\cdots+\lambda_{d})^{\alpha}},
\end{align*}
where $\alpha>0$ and $\beta=(\beta_{1},\ldots,\beta_{d})$.
The reason of the prior choice is as follows.
The optimization in finding the MAP estimate based on this prior is a bit tricky due to the singularity around $\lambda=0$, while 
we can easily access the posterior expectation 
as the efficient Gibbs sampling algorithm is available.
So, the matching prior pair can offer useful surrogates of the MAP estimate.
The number of dimension $d$ is $100$ for synthetic data analysis and is $99$ for real data analysis, respectively.
We set $\beta=(3,\ldots,3)$ and $\alpha=\sum_{j=1}^{d}\beta_{j}-1$.
In order to avoid the singularity issue in the optimization finding the MAP estimate, 
we restrict the parameter space to $[10^{-3},\infty)^{d}$ on the basis of the try-and-error.
}

We begin with displaying the numerical experiment using the following synthetic data:
\revise{
\begin{align*}
Y(t)=\begin{pmatrix}
Y_{1}(t)\\
Y_{2}(t)\\
\vdots\\
Y_{100}(t)
\end{pmatrix} \quad \mid \quad \lambda =  \begin{pmatrix}
\lambda_{1}\\
\lambda_{2}\\
\vdots\\
\lambda_{100}
\end{pmatrix}
\quad&\sim\quad
\otimes_{i=1}^{100} \mathrm{Poisson}(\lambda_{i})
\quad (t=1,\ldots,n),\\
\lambda_{j}&= 
\begin{cases}
0.001 &\text{if}\,\,j \,\,\text{is odd},\\
 2 &\text{if otherwise}
\end{cases}
\quad
(j=1,\ldots,100).
\end{align*}
}
In this experiment, we display the result for one realization because the result is not so much dependent on realization.
For the calculation of the posterior mean, we use 10000 MCMC samples.

Figure \ref{fig:Poisson_syn} showcases the MAP estimate based on the shrinkage prior (colored in black), 
the posterior mean based on the shrinkage prior (colored in red), 
and the posterior mean based on the matching prior pair (colored in blue).
From \revise{(c)-(d)} of Figure \ref{fig:Poisson_syn},
we see that the posterior mean based on the matching prior pair of the shrinkage prior can get closer to the MAP estimate based on the shrinkage prior than that based on the shrinkage prior.
Surprisingly, even for high dimensional cases such as $(n,d)=(1,100)$ and $(n,d)=(10,100)$,
the matching prior pair works well.
One of potential reasons for this success in high dimension is that a Laplace approximation of the posterior distribution might still work in certain high-dimensional set-ups (e.g., \citealp{panov2015,yano2020,Kasprzaki2023}).
\revise{Further,
we measure the computational time (CPU times) for the MAP estimate and the posterior expectation based on the matching prior pair.
In this example, due to the singularity issue and the high-dimensionality,
the optimization for the MAP estimate
is relatively slow compared to the MCMC algorithm as in Table \ref{tab: computational time Poisson}. Thus, this implies that the Bayesian computation using the matching prior pair can offer a good surrogate for the MAP estimate if we have an efficient MCMC algorithm and
the optimization is slow or difficult.
}
\begin{table}[h]
\label{tab: computational time Poisson}
\caption{\revise{Computational time for obtaining the MAP estimates and the posterior expectation based on the matching prior pair. The set-up is the same as in Figure \ref{fig:Poisson_syn}.}}
\begin{center}
\begin{tabular}{|r|c|c|}
\hline
  Sample size $n$ & MAP & Posterior measn based on Matching prior pair \\\hline
  $1$ & 6.38 s & 0.07 s\\
  $10$ & 3.53 s& 0.06 s\\
  $100$ & 5.12 s& 0.06 s\\
  $1000$ & 12.5 s& 0.06 s\\ \hline
\end{tabular}
\end{center}
\end{table}

We proceed to an application to Japanese pickpocket data from \cite{pickpocket}. This data reports the total numbers of pickpockets in each year in Tokyo Prefecture, and are classified by town and also by the type of crimes. We use pickpocket data from \revise{2012 to 2013} at 99 towns in Chuo ward. 
We work with the Poisson sequence model ($d=99$, $n\le 2$; $n$ is the number of years we use in the analysis) and report how the matching prior pair calibrates the shrinkage prior so as to get the posterior mean closer to the MAP estimate based on the shrinkage prior. For the calculation of the posterior mean, we use 10000 MCMC samples.

Figure \ref{fig:Poisson_real_chuo} showcases the MAP estimate based on the shrinkage prior $\pi$ (colored in black), the posterior mean based on the shrinkage prior $\pi$ (colored in red), and the posterior mean based on the matching prior pair of $\pi$ (colored in blue) for pickpockets in Chuo ward, Tokyo, Japan.
Figure \ref{fig:Poisson_real_chuo} shows that for pickpocket data,
the MAP estimate and the posterior mean based on the same shrinkage prior are different although the difference gets smaller as the sample size becomes larger,
and
the matching prior pair successfully yields the posterior mean closer to the MAP estimate based on the improper shrinkage prior.

\section{Proofs}
\label{sec:proofs}

This section provides the proof of the main results.

\begin{proof}[Proof of Theorem \ref{thm_match}]

The proof employs the following asymptotic expansions for a posterior mean and a MAP estimate. 
\begin{lem}
\label{prop_post}
%%%
The posterior mean of $\theta$ based on a prior $\pi_{\rm PM}(\theta)$ is expanded as
\begin{align}
\label{eq:post}
\hat{\theta}_{\rm PM}^a
=
\hat{\theta}_{\rm MLE}^a
&+\frac{{g}^{ab}(\hat{\theta}_{\rm MLE})}{n}
\left(\partial_b\log\frac{\pi_{\rm PM}}{{\pi}_{\rm J}}(\hat{\theta}_{\rm MLE})
+\frac{T_b(\hat{\theta}_{\rm MLE})}{2}\right)
\nonumber\\
&
+\frac{{g}^{bc}(\hat{\theta}_{\rm MLE})}{2n}\left(-\mGamma{bc}{~a}(\hat{\theta}_{\rm MLE})\right)
+\mathrm{o}_p(n^{-1})
\quad(a=1,\ldots,d).
\end{align}
\end{lem}

\begin{lem}
\label{prop_map}
%%%
The MAP estimate of $\theta$ based on a prior $\pi_{\rm MAP}(\theta)$ is expanded as
\begin{align}
\label{eq:map}
\hat{\theta}_{\rm MAP}^a
=
\hat{\theta}_{\rm MLE}^a
+\frac{{g}^{ab}(\hat{\theta}_{\rm MLE})}{n}
\partial_b\log{\pi_{\rm MAP}}(\hat{\theta}_{\rm MLE})
+\mathrm{o}_p(n^{-1})
\quad (a=1,\ldots,d).
\end{align}
\end{lem}

The proofs of these lemmas are given right after the main proof.
%The proof of Lemma \ref{prop_post} proceeds following the proof of Theorem I\hspace{-1.2pt}I\hspace{-1.2pt}I.1 in \cite{okudo2021} by replacing $\eta$ by $\theta$, and the proof of Lemma \ref{prop_map} closely follows Section 7.3 of \cite{mccullagh2018}.

These lemmas give the following condition under which
the posterior mean $\hat{\theta}_{{\rm PM}}$ (\ref{eq:post}) and the MAP estimate $\hat{\theta}_{\rm MAP}$ (\ref{eq:map}) coincide except for $o_p(n^{-1})$ terms:
for $a=1,\ldots,d$,
\begin{align*}
&\frac{{g}^{ab}(\hat{\theta}_{\rm MLE})}{n}
\left(\partial_b\log\frac{\pi_{\rm PM}}{{\pi}_{\rm J}}(\hat{\theta}_{\rm MLE})+\frac{T_b(\hat{\theta}_{\rm MLE})}{2}\right)
+\frac{{g}^{bc}(\hat{\theta}_{\rm MLE})}{2n}\left(-\mGamma{bc}{~a}(\hat{\theta}_{\rm MLE})\right)\\
&=
\frac{{g}^{ab}(\hat{\theta}_{\rm MLE})}{n}
\partial_b\log{\pi_{\rm MAP}}(\hat{\theta}_{\rm MLE})+\mathrm{o}_p(n^{-1}).
\end{align*}
This is rewritten as
\begin{align}
&{{g}^{ab}(\hat{\theta}_{\rm MLE})}
\partial_b\log\frac{\pi_{\rm PM}}{{\pi}_{\rm MAP}}
(\hat{\theta}_{\rm MLE})
\nonumber\\
&= {g}^{ab}(\hat{\theta}_{\rm MLE})\left(
\partial_b\log{\pi_{\rm J}}(\hat{\theta}_{\rm MLE})
- \frac{T_b(\hat{\theta}_{\rm MLE})}{2}\right) 
+ \frac{{g}^{bc}(\hat{\theta}_{\rm MLE})}{2}\mGamma{bc}{~a}(\hat{\theta}_{\rm MLE})+\mathrm{o}_p(n^{-1})\nonumber\\
&= {g}^{ab}(\hat{\theta}_{\rm MLE})\left[
\partial_b\log{\pi_{\rm J}}
(\hat{\theta}_{\rm MLE})- \frac{1}{2}\left\{\mGamma{cdb}{}(\hat{\theta}_{\rm MLE})-\eGamma{cdb}{}(\hat{\theta}_{\rm MLE})\right\}g^{cd}(\hat{\theta}_{\rm MLE})\right] \nonumber\\
&\quad + \frac{{g}^{bc}(\hat{\theta}_{\rm MLE})}{2}\mGamma{bc}{~a}(\hat{\theta}_{\rm MLE})+\mathrm{o}_p(n^{-1})\nonumber\\
&= {g}^{ab}(\hat{\theta}_{\rm MLE})
\partial_b\log{\pi_{\rm J}}
(\hat{\theta}_{\rm MLE})+ \frac{{g}^{bc}(\hat{\theta}_{\rm MLE})}{2}\eGamma{bc}{~a}(\hat{\theta}_{\rm MLE})+\mathrm{o}_p(n^{-1}),
\label{eq: geometric condition}
\end{align}
where the second identity follows since $T_{abc}=\mGamma{abc}{}-\eGamma{abc}{}$.
This completes the proof.
%%
%Using the dual structure (\ref{eq: dual structure}) of e-\&m-connections 
%
%\[
%\partial_c g_{de} = \eGamma{cde}{} + \eGamma{dec}{} + T_{ecd}
%\]
%
%and the identity for the Jeffreys prior as the 0-parallel prior
%
%\[
%g^{cd}\partial_b g_{cd} = 2\partial_b \log\pi_J,
%\]
%
%we have
%
%\begin{align*}
%g^{ab}\partial_b \log\pi_J 
%&= \frac{1}{2}g^{ab}g^{cd}\partial_b g_{cd}\\
%&= \frac{1}{2}g^{ab}g^{cd}(\eGamma{bcd}{} + \eGamma{cdb}{} + T_{bcd})\\
%&= \frac{1}{2}g^{ab}g^{cd}(\eGamma{cdb}{} + \mGamma{bcd}{}),
%end{align*}
%
%which gives the identity
%
%\begin{align*}
%    \frac{{g}^{cd}}{2}\eGamma{cd}{~a}
%    &= \frac{1}{2}g^{ab}g^{cd}\eGamma{cdb}{} - \frac{1}{2}g^{ab}g^{cd}(\eGamma{cdb}{} + \mGamma{bcd}{}) + g^{ab}\partial_b \log\pi_J \\
%    &= -\frac{1}{2}g^{ab}\textcolor{green}{\mGamma{bc}{c}} + g^{ab}\partial_b \log\pi_J.
%\end{align*}
%
%Putting this identity to (\ref{eq: geometric condition}) yields
%
%\begin{align*}
%    {{g}^{ab}(\hat{\theta}_{\rm MLE})}
%    \partial_b\log\frac{\pi_{\rm PM}}{{\pi}_2}(\hat{\theta}_{\rm MLE})
%    &= {g}^{ab}(\hat{\theta}_{\rm MLE})
%    \partial_b\log{\pi_J}(\hat{\theta}_{\rm MLE})
%    + \frac{{g}^{bc}(\hat{\theta}_{\rm MLE})}{2}\eGamma{bc}{~a}(\hat{\theta}_{\rm MLE})+\mathrm{o}_p(n^{-1})\\
%    &= -\frac{1}{2}g^{ab}(\hat{\theta}_{\rm MLE})\eGamma{bc}{c}(\hat{\theta}_{\rm MLE}) + 2g^{ab}(\hat{\theta}_{\rm MLE})\partial_b \log\pi_J(\hat{\theta}_{\rm MLE})+\mathrm{o}_p(n^{-1}),
%\end{align*}
%which completes the proof.
%
\end{proof}

\begin{proof}[Proof of Lemma \ref{prop_post}]

In the proof, we consider an approximation of the posterior expectation of arbitrary third-times differentiable function $f:\Theta\to\mathbb{R}$.
Setting $f(\theta)=\theta^{a}$ ($a=1,\ldots,d$) gives the approximation of the posterior mean of $\theta$.
The following proof of Lemma \ref{prop_post} proceeds closely following the proof of Theorem I\hspace{-1.2pt}I\hspace{-1.2pt}I.1 in \cite{okudo2021}. 
The first step is to employ the Laplace approximation of integrals to get an approximation of the posterior expectation.
The second step is to arrange terms in information-geometrical notations.
% by replacing $\eta$ by $\theta$.
%but we write the proof here for completeness.

\textbf{Step 1: Laplace approximation}.
Observe that the posterior expectation of a third-times differentiable function $f(\theta)$ based on a prior $\pi(\theta)$ is written as
\begin{align*}
f_{\pi}(y^n)=\frac{\int f(\theta) p(y^n\,;\,\theta)\pi(\theta)d\theta}{\int p(y^n\,;\,\theta)\pi(\theta)d\theta}
= \frac{\int f(\theta) \exp(n\bar{L}(\theta))\pi(\theta)d\theta}{\int \exp(n\bar{L}(\theta)) \pi(\theta)d\theta},
\end{align*}
where $\bar{L}(\theta)=(1/n)\sum_{t=1}^{n}\log p(y(t)\,;\,\theta)$.
We approximate this using the Laplace method (e.g., Theorem  4.6.1 of \cite{kass1997}
%, Sec. 3.6 and Sec. 4.6, 
and \cite{tierney1986}).
%The proposed approximation follows Theorem  4.6.1 in  \cite{kass1997}.
Consider an expansion of $\pi(\theta)\exp(n\bar{L}(\theta))$ around $\theta=\hat{\theta}_{\rm MLE}$.
In the following, 
for any function $g(\theta)$,
we abbreviate the value $g(\hat{\theta}_{\rm MLE})$ to $\hat{g}$;
e.g., $\hat{\pi}:=\pi(\hat{\theta}_{\rm MLE})$.
By rescaling $\theta$ as $\theta = \hat{\theta}_{\rm MLE} +\phi/\sqrt{n}$, 
we get
\begin{align}
&\pi(\theta)\exp(n\bar{L}(\theta))\nonumber\\
&=
\left(
\hat{\pi} + \frac{(\partial_a\hat{\pi}) \phi^a}{\sqrt{n}} + \frac{(\partial_{ab}\hat{\pi})\phi^a\phi^b}{2n}
+ \frac{(\partial_{abc}\hat{\pi})\phi^a\phi^b\phi^c}{6n\sqrt{n}}+\mathrm{O}_p(n^{-2})
\right)\nonumber\\
&~~~\times\exp
\left(
n\hat{L} +\frac{(\partial_{ab}\hat{L})\phi^a\phi^b}{2}
+ \frac{(\partial_{abc}\hat{L})\phi^a\phi^b\phi^c}{6\sqrt{n}} + \frac{(\partial_{abcd}\hat{L})\phi^a\phi^b\phi^c\phi^d}{24{n}} 
\right.\nonumber\\
&~~~~~~~~~~~~~~~~~~~~~~~~~~~~~~~~~~~~~~~~~~~~~~~~~~~~~~~~~~~
\left.+\frac{(\partial_{abcde}\hat{L})\phi^a\phi^b\phi^c\phi^d\phi^e}{120n\sqrt{n}}+\mathrm{O}_p(n^{-2})
\right)\nonumber\\
&=
\hat{\pi}e^{n\hat{L}}e^{(\partial_{ab}\hat{L})\phi^a\phi^b/2} \left(
1 + \frac{(\partial_a \hat{\pi}) \phi^a}{\hat{\pi}\sqrt{n}} + \frac{(\partial_{ab}\hat{\pi})\phi^a\phi^b}{2 \hat{\pi}n}
 + \frac{(\partial_{abc}\hat{\pi})\phi^a\phi^b\phi^c}{6 \hat{\pi}n\sqrt{n}}+\mathrm{O}_p(n^{-2})
\right)\nonumber\\
&~~~\times
\left(
1 +  \frac{(\partial_{abc}\hat{L})\phi^a\phi^b\phi^c}{6\sqrt{n}} + \frac{(\partial_{abcd}\hat{L})\phi^a\phi^b\phi^c\phi^d}{24{n}}  
\right.\nonumber\\
&~~~~~~~~~~~~~~~~~~~~~~~~~~~~~~~~~~~~~ 
\left. + \frac{(\partial_{abc}\hat{L})(\partial_{a'b'c'}\hat{L})\phi^a\phi^b\phi^c\phi^{a'}\phi^{b'}\phi^{c'}}{72{n}}
+\mathrm{O}_p(n^{-3/2})
\right)\nonumber\\
&=
\hat{\pi}e^{n\hat{L}}e^{-\hat{J}_{ab}\phi^a\phi^b/2}\left(
1+\frac{(\partial_a \hat{\pi}) \phi^a}{\hat{\pi}\sqrt{n}} 
{+\frac{(\partial_{abc}\hat{L})\phi^a\phi^b\phi^c}{6\sqrt{n}}} + \frac{(\partial_{ab}\hat{\pi})\phi^a\phi^b}{2\hat{\pi}n}
\right.\nonumber\\
&~~~~~~\left.
+ \frac{(\partial_a \hat{\pi}) (\partial_{bcd}\hat{L})\phi^a\phi^b\phi^c\phi^d}{6\hat{\pi}n}\right.\nonumber\\
&~~~~~~\left. + \frac{(\partial_{abc}\hat{\pi})\phi^a\phi^b\phi^c}{6\hat{\pi}n\sqrt{n}}
+ \frac{(\partial_{ab}\hat{\pi})(\partial_{cde}\hat{L})\phi^a\phi^b\phi^c\phi^d\phi^e}{12\hat{\pi}n\sqrt{n}}
+ \frac{(\partial_a \hat{\pi})(\partial_{bcde}\hat{L}) \phi^a\phi^b\phi^c\phi^d\phi^e}{24\hat{\pi}n\sqrt{n}}\right.\nonumber\\
&~~\left. + \frac{(\partial_a \hat{\pi})(\partial_{bcd}\hat{L})(\partial_{efg}\hat{L}) \phi^a\phi^b\phi^c\phi^d\phi^e\phi^f\phi^g}{72\hat{\pi}n\sqrt{n}}
{+ \frac{C_1}{n}} + \mathrm{O}_p(n^{-2})
\right),
\label{eq: asymptotic expansion of piexp}
\end{align}
where the second equation follows from $\exp(x)=1+x+O(x^{2})$,
and in the third equation, we denote $-\partial_{ab}\hat{L}$ by
$\hat{J}_{ab}$ and
denote terms not depending on $\pi$ and $n$ by $C_1$,
respectively.

Next, we integrate both sides of (\ref{eq: asymptotic expansion of piexp}) with respect to $\theta$.
Let $(\hat{J}^{ab})$ be the inverse matrix of $(\hat{J}_{ab})$.
By changing the variables from $\theta$ to $\phi$, and by using the formula of moments of multivariate Gaussian distributions, we obtain
\begin{align*}
&\int \pi(\theta)\exp(n\bar{L}(\theta)) d\theta\\
&=
C_2\hat{\pi}\left(
1
+ \frac{(\partial_{ab}\hat{\pi})}{2\hat{\pi}n}
\int \phi^a\phi^b e^{-\hat{J}_{cd}\phi^c\phi^d/2} d\phi
+ \frac{(\partial_a \hat{\pi}) (\partial_{bcd}\hat{L})}{6\hat{\pi}n}
\int \phi^a\phi^b\phi^c\phi^d e^{-\hat{J}_{ef}\phi^e\phi^f/2} d\phi
\right.\\
&\left.\quad\quad\quad\quad + \frac{C_1}{n} 
+ \mathrm{O}_p(n^{-2})
\right)
\\
&= C_2\hat{\pi}
\left(
1+\frac{(\partial_{ab}\hat{\pi})\hat{J}^{ab}}{2\hat{\pi}n}
+ \frac{(\partial_a \hat{\pi}) (\partial_{bcd}\hat{L})(\hat{J}^{ab}\hat{J}^{cd}+\hat{J}^{ac}\hat{J}^{bd}+\hat{J}^{ad}\hat{J}^{bc})}{6\hat{\pi}n} + {\frac{C_1}{n}} 
+ \mathrm{O}_p(n^{-2}) \right)\\
&= C_2\hat{\pi}
\left(
1+\frac{(\partial_{ab}\hat{\pi})\hat{J}^{ab}}{2\hat{\pi}n}+ \frac{(\partial_a \hat{\pi}) (\partial_{bcd}\hat{L})\hat{J}^{ab}\hat{J}^{cd}}{2\hat{\pi}n} + {\frac{C_1}{n}}
+ \mathrm{O}_p(n^{-2})
\right),
\end{align*}
where $C_2$ is a constant not depending on $\pi$ and $n$.
Replacing $\pi(\theta)$ by $f(\theta)\pi(\theta)$ for an arbitrary third-times differentiable function $f:\mathbb{R}^d\to \mathbb{R}$, 
we have
\begin{align*}
&\int f(\theta) \pi(\theta)\exp(n\bar{L}(\theta)) d\theta\\
&= C_2\hat{f}\hat{\pi}
\left(
1+\frac{\{\partial_{ab}(\hat{f}\hat{\pi})\}\hat{J}^{ab}}{2\hat{f}\hat{\pi}n}
+ \frac{\{\partial_a(\hat{f}\hat{\pi})\} (\partial_{bcd}\hat{L})\hat{J}^{ab}\hat{J}^{cd}}{2\hat{f}\hat{\pi}n}
+ {\frac{C_1}{n}}
 + \mathrm{O}_p(n^{-2})
\right).
\end{align*}
Therefore, the posterior expectation of $f(\theta)$ is expanded as
\begin{align}
f_{\pi}&=\frac{\int f(\theta) \exp(n\bar{L}(\theta))\pi(\theta)d\theta}{\int \exp(nL(\theta)) \pi(\theta)d\theta}\nonumber\\
&=
\left.C_2\hat{f}\hat{\pi}
\left(
1+\frac{\partial_{ab}(\hat{f}\hat{\pi})\hat{J}^{ab}}{2\hat{f}\hat{\pi}n}
+ \frac{\{\partial_a(\hat{f}\hat{\pi})\} (\partial_{bcd}\hat{L})\hat{J}^{ab}\hat{J}^{cd}}{2\hat{f}\hat{\pi}n} + {\frac{C_1}{n}}
 + \mathrm{O}_p(n^{-2})
\right) \right.\nonumber\\
&~~~\left. \middle/ 
C_2\hat{\pi}
\left(
1+\frac{(\partial_{ab}\hat{\pi})\hat{J}^{ab}}{2\hat{\pi}n} + \frac{(\partial_a \hat{\pi}) (\partial_{bcd}\hat{L})\hat{J}^{ab}\hat{J}^{cd}}{2\hat{\pi}n}
 + {\frac{C_1}{n}}
 + \mathrm{O}_p(n^{-2})
\right) \right.\nonumber\\
&= \hat{f}
\left(
1+\frac{\hat{J}^{ab}}{2n}\left(\frac{\partial_{ab}(\hat{f}\hat{\pi})}
{\hat{\theta}_i\hat{\pi}}-\frac{\partial_{ab}\hat{\pi}}{\hat{\pi}} \right)
+\frac{\hat{J}^{ab}\hat{J}^{cd}\partial_{bcd}\hat{L}}{2n}
\left(  \frac{\partial_a(\hat{f}\hat{\pi})}{\hat{\theta}\hat{\pi}}-\frac{\partial_a \hat{\pi}}{\hat{\pi}} \right) +\mathrm{O}_p(n^{-2})
\right)\nonumber\\
&= \hat{f} +
\frac{\hat{J}^{ab}}{2n}\left(\partial_{ab}\hat{f} +
\frac{2 (\partial_a \hat{f}) (\partial_b \hat{\pi})}{\hat{\pi}} \right)
+\frac{\hat{J}^{ab}\hat{J}^{cd}\partial_{bcd}\hat{L}}{2n} \partial_a \hat{f}
+\mathrm{O}_p(n^{-2}).
\label{eq: Laplace approximation of posterior expectation}
\end{align}
This completes Step 1.

\textbf{Step 2: Rearrangement using the information-geometric notations}.
The law of large numbers yields 
$\hat{J}_{ab}=\hat{g}_{ab}+o_p(1)$ and 
$\partial_{bcd}\hat{L}=\mathrm{E}_{\theta}[\partial_{bcd}\hat{L}]+o_p(1)$.
The Bartlett identity gives
\begin{align}
\label{eq:bartlett}
\mathrm{E}_{\theta}[\partial_{bcd}\bar{L}] = -\partial_b g_{cd} (\theta) -\eGamma{cdb}{}(\theta)=-\partial_b g_{cd} (\theta)-\mGamma{cdb}{}(\theta)+T_{bcd}(\theta).
\end{align}
Together with the definition of 0-parallel prior
$g^{cd}\partial_b g_{cd}=\partial_b \log(|g|)=2\partial_b \log \pi_{\rm J}$,
these give the following representation of the approximated posterior expectation of $f$:
\revise{
\begin{align*}
f_{\pi}
&= \hat{f} +
\frac{\hat{g}^{ab}}{2n}\left( \partial_{ab}\hat{f} + 2 \partial_a \hat{f} \partial_b \log \hat{\pi} \right)\\
&\quad\quad +\frac{\hat{g}^{ab}}{2n}\left( -2\partial_b \log \hat{\pi}_J - \hat{g}^{cd}\mGamma{cdb}{}(\hat{\theta}_{\rm MLE})+\hat{T}_{b}
\right) \partial_a \hat{f}
+\mathrm{o}_p(n^{-1}) \\
&=\hat{f} +
\frac{\hat{g}^{ab}}{2n}
\left( \partial_{ab}\hat{f} - \mGamma{ab}{~c}(\hat{\theta}_{\rm MLE}) \partial_c \hat{f} \right)
+\frac{\hat{g}^{ab}}{n}
\left( \partial_b \log \frac{\hat{\pi}}{\hat{\pi}_J} + \frac{\hat{T}_{b}}{2}
\right)\partial_a \hat{f}
+\mathrm{o}_p(n^{-1}).
\end{align*}
}
Thus, replacing $f$ by $\theta^{a}$, we have
\begin{align*}
\hat{\theta}^a_{\rm PM}
&=\hat{\theta}^a +
\frac{\hat{g}^{bc}}{2n}
\left(- \mGamma{bc}{~a}(\hat{\theta}_{\rm MLE})\right)
+\frac{\hat{g}^{ab}}{n}
\left( \partial_b \log \frac{\hat{\pi}}{\hat{\pi}_J} + \frac{T_{b}(\hat{\theta}_{\rm MLE})}{2}
\right)
+\mathrm{o}_p(n^{-1}).
\end{align*}

\end{proof}

\begin{proof}[Proof of Lemma \ref{prop_map}]
%The proof of Lemma \ref{prop_map} closely follows Section 7.3 of \cite{mccullagh2018}.
Observe the definition of 
the MAP estimate $\hat{\theta}_{\rm MAP}$:
\begin{align*}
    n \partial_a \bar{L}(\hat{\theta}_{\rm MAP}) + \partial_a \log \pi(\hat{\theta}_{\rm MAP}) = 0 \,\,(a=1,\ldots,d),
\end{align*}
where $\bar{L}(\theta)=(1/n)\sum_{t=1}^{n}\log p(y(t)\,;\,\theta)$.
Letting $\delta=\hat{\theta}_{\rm MAP}-\hat{\theta}_{\rm MLE}$,
the Taylor expansion around $\hat{\theta}_{\rm MLE}$ yields, for $a=1,\ldots,d$,
\begin{align*}
    0 &= \partial_a \bar{L}(\hat{\theta}_{\rm MAP}) + \frac{1}{n}\partial_a \log \pi(\hat{\theta}_{\rm MAP})\\
     &= \partial_a \bar{L}(\hat{\theta}_{\rm MLE}) + \delta^b \partial_{ab}\bar{L}(\hat{\theta}_{\rm MLE}) %+ \frac{1}{2}\delta^b \delta^c \partial_{abc}\hat{L}
    + \frac{1}{n}\partial_a \log \pi (\hat{\theta}_{\rm MLE}) + {\rm O}_p(\|\delta\|^{2}),\\
     &= \delta^b \partial_{ab}\bar{L}(\hat{\theta}_{\rm MLE}) %+ \frac{1}{2}\delta^b \delta^c \partial_{abc}\hat{L}
    + \frac{1}{n}\partial_a \log \pi(\hat{\theta}_{\rm MLE}) + {\rm O}_p(\|\delta\|^{2}),
\end{align*}
where the last equation follows since $\partial_{a}\bar{L}(\hat{\theta}_{\rm MLE})=0$ for $a=1,\ldots,d$. 
Because the law of large numbers and the central limit theorem give 
\[
\partial_{ab}\bar{L}(\hat{\theta}_{\rm MLE}) = -g_{ab}(\hat{\theta}_{\rm MLE}) + {\rm o}_p(1)
\quad\text{and}\quad
\partial_{ab}\bar{L}(\hat{\theta}_{\rm MLE}) + g_{ab}(\hat{\theta}_{\rm MLE}) = {\rm O}_p(1/\sqrt{n}),
\]
we get
\begin{align*}
         \delta^b g_{ab}(\hat{\theta}_{\rm MLE})
         +{\rm O}_{p}(\|\delta\| /\sqrt{n})
         = \frac{1}{n}\partial_a \log\pi(\hat{\theta}_{\rm MLE}) + {\rm O}_p(\|\delta\|^{2}),\\
\end{align*}
which yields
\begin{align*}
        \delta^b &= \frac{1}{n} g^{ab} (\hat{\theta}_{\rm MLE})  \partial_a \log \pi(\hat{\theta}_{\rm MLE}) + {\rm o}_p(n^{-1})
        \,\,(b=1,\ldots,d)
\end{align*}
and completes the proof.
\end{proof}

\revise{
\begin{proof}[Proof of Corollary \ref{cor_calibration}]
Using (\ref{eq:bartlett}) and the definition of 0-parallel prior,
\begin{align*}
    g^{ab}g^{cd}\mathrm{E}_{\theta}[\partial_{bcd}\bar{L}] 
&=-g^{ab}g^{cd}\partial_b g_{cd} (\theta)-g^{ab}g^{cd}\mGamma{cdb}{}(\theta)+g^{ab}g^{cd}T_{bcd}(\theta)\\
&= -2g^{ab}\partial_b \log\pi_\mathrm{J} (\theta)-g^{cd}\mGamma{cd}{a}(\theta)+g^{ab}T_{b}(\theta).
\end{align*}
From (\ref{eq:post}) and (\ref{eq:map}), we have for $a=1,\ldots,d$,
\begin{align*} 
&\hat{\theta}_{\rm PM}^a - \hat{\theta}_{\rm MAP}^a \\
&=
\frac{{g}^{ab}(\hat{\theta}_{\rm MLE})}{n}
\left(\partial_b\log\frac{\pi_{\rm PM}}{{\pi}_{\rm J}}(\hat{\theta}_{\rm MLE})
+\frac{T_b(\hat{\theta}_{\rm MLE})}{2}\right)
+\frac{{g}^{bc}(\hat{\theta}_{\rm MLE})}{2n}\left(-\mGamma{bc}{~a}(\hat{\theta}_{\rm MLE})\right) \\
&~~~ -\frac{{g}^{ab}(\hat{\theta}_{\rm MLE})}{n}
\partial_b\log{\pi_{\rm MAP}}(\hat{\theta}_{\rm MLE})
+\mathrm{o}_p(n^{-1})\\
&= \frac{{g}^{ab}(\hat{\theta}_{\rm MLE})}{n}\partial_b\log\frac{\pi_{\rm PM}}{{\pi}_{\rm MAP}}(\hat{\theta}_{\rm MLE}) + \frac{1}{2n}g^{ab}(\hat{\theta}_{\rm MLE})g^{cd}(\hat{\theta}_{\rm MLE})\mathrm{E}_{\theta}[\partial_{bcd}\bar{L}(\hat{\theta}_{\rm MLE})]+\mathrm{o}_p(n^{-1})\\
&= \frac{{g}^{ab}(\hat{\theta}_{\rm MAP})}{n}\partial_b\log\frac{\pi_{\rm PM}}{{\pi}_{\rm MAP}}(\hat{\theta}_{\rm MAP}) + \frac{1}{2n}g^{ab}(\hat{\theta}_{\rm MAP})g^{cd}(\hat{\theta}_{\rm MAP})\partial_{bcd}\bar{L}(\hat{\theta}_{\rm MAP})+\mathrm{o}_p(n^{-1}).\\
\end{align*}
In the last identity, we used $\partial_{bcd}\bar{L}(\hat{\theta}_{\rm MLE}) = \mathrm{E}_{\theta}[\partial_{bcd}\bar{L}(\hat{\theta}_{\rm MLE})]+\mathrm{o}_p(1)$ and $\hat{\theta}_{\rm MLE}=\hat{\theta}_{\rm MAP}+\mathrm{o}_p(1)$.
When $\pi_{\rm PM} = {\pi}_{\rm MAP}$, we have
\begin{align*}
    \hat{\theta}_{\rm PM}^a - \hat{\theta}_{\rm MAP}^a
&= \frac{1}{2n}g^{ab}(\hat{\theta}_{\rm MAP})g^{cd}(\hat{\theta}_{\rm MAP})\partial_{bcd}\bar{L}(\hat{\theta}_{\rm MAP})+\mathrm{o}_p(n^{-1})\\
&= \frac{1}{2n^2}g^{ab}(\hat{\theta}_{\rm{MAP}})g^{cd}(\hat{\theta}_{\rm{MAP}})
\sum_{t=1}^{n}\partial_{bcd}\log p(y(t)\,;\,\hat{\theta}_{\rm{MAP}})+\mathrm{o}_p(n^{-1}).
\end{align*}
\end{proof}
}

\begin{proof}[Proof of Proposition \ref{prop_moment_general}]

This proposition is proved simply by changing Lemmas \ref{prop_post}-\ref{prop_map} to the following lemmas.

\begin{lem}
\label{prop_post_general}
For $i=1,\ldots,d$,
%%%
the posterior mean of $f_{i}(\theta)$ based on a prior $\pi(\theta)$ is expanded as
\begin{align*}
\label{eq:post}
(f_\pi)_i
=&
f_i(\hat{\theta}_{\rm MLE})
+\frac{{g}^{ab}(\hat{\theta}_{\rm MLE})}{n}
\left(\partial_b\log\frac{\pi}{{\pi}_J}(\hat{\theta}_{\rm MLE})+\frac{T_b(\hat{\theta}_{\rm MLE})}{2}\right)\partial_a f_i(\hat{\theta}_{\rm MLE})\\
&+\frac{{g}^{bc}(\hat{\theta}_{\rm MLE})}{2n}\left(\partial_b\partial_c f_i(\hat{\theta}_{\rm MLE})-\mGamma{bc}{~a}(\hat{\theta}_{\rm MLE})\partial_a f_i(\hat{\theta}_{\rm MLE})\right)
+\mathrm{o}_p(n^{-1}).
\end{align*}
\end{lem}

\begin{lem}
\label{prop_map_general}
%%%
For $i=1,\ldots,d$,
a plugin of the MAP estimate $\hat{\theta}_{\rm MAP}$ of $\theta$ based on a prior $\pi(\theta)$ into a statistic $f_{i}(\theta)$ is expanded as
\begin{align*}
f_i(\hat{\theta}_{\rm MAP})
=
f_i(\hat{\theta}_{\rm MLE})
+\frac{{g}^{ab}(\hat{\theta}_{\rm MLE})}{n}
\partial_b\log{\pi}(\hat{\theta}_{\rm MLE})\partial_a  f_i(\hat{\theta}_{\rm MLE})
+\mathrm{o}_p(n^{-1}).
\end{align*}
\end{lem}
The proofs of these lemmas are straightforward and omitted.

\end{proof}

\section{Acknowledgements}

\revise{The authors thank the Editor, the handling editor, and two referees for their constructive comments that have improved the quality of this paper.
The authors thank Kaoru Irie for helpful comments to the early version of this work.
The authors thank Ryoya Kaneko for sharing his python codes of data preprocessing.
The authors thank Takemi Yanagimoto for helpful discussions.}
This work is supported by JSPS KAKENHI (JP19K20222, JP20K23316, JP21H05205, JP21K12067, JP22H00510, JP23K11024), MEXT (JPJ010217), and ``Strategic Research Projects'' grant (2022-SRP-13) from ROIS (Research Organization of Information and Systems).

\revise{
Some of the python codes used in this paper are available at \url{https://doi.org/10.5281/zenodo.13854194}.}

\section*{Supplement}

\revise{
This supplement provides regularity conditions used in this paper and the validity of Corollary \ref{cor_calibration}.}

\section*{Regularity Conditions}
%For random variables $Y_n$, denote $Y_n = t_n$ if for all $\epsilon > 0$, ${P}(|Y_n| > \epsilon) = \mathrm{o}(n^{-2})$.
\revise{
Let $\pi({\theta}\mid y^n)$ denote the posterior density of $\theta$ given observation $y^n$.
We assume the following regularity conditions in \cite{hartigan1998}:
%\sdescription{Short description of Supplement A.}
\begin{itemize}
    \item[A1.] All derivatives of $\log p(y;\theta)$ up to the fifth order exist for all $y$ and ${\theta}$ in a neighborhood of the true parameter ${\theta}_0$.
    \item[A2.] All moments exist for the first four derivatives, as well as for the maximum squared fifth derivatives in a neighborhood of ${\theta}_0$. Additionally, the moments are differentiable in a neighborhood of ${\theta}_0$.
    \item[A3.] The information matrix whose $(a,b)$ component is $\mathrm{E}_{\theta_0}[\partial_a \log p(y^n;\theta) \partial_b \log p(y^n;\theta)]$ is positive definite.
    \item[A4.] $\mathrm{E}_{\theta_0}[\partial_{a_1}\partial_{a_2}\dots \partial_{a_r}  p(y;\theta)/p(y;\theta)]_{\theta=\theta_0} = 0, \text{ for } r=1,2,\dots,4$.
    \item[A5.] The prior density $\pi$ is positive and possesses two derivatives in a neighborhood of ${\theta}_0$.
    \item[A6.] %$|\hat{{\theta}}_\mathrm{MLE} - {\theta}_0|$ satisfies
    For all $\epsilon > 0$,
    $\mathrm{P}[\|\hat{{\theta}}_\mathrm{MLE} - {\theta}_0\| > \epsilon] = \mathrm{o}(n^{-2})$.
    %$|\hat{{\theta}}_\mathrm{MLE} - {\theta}_0| = t_n$.
    \item[A7.] %Posterior tail probabilities are negligible: 
    %$\mathrm{P}[\int_{|{\theta} - {\theta}_0| > \epsilon} \sup_{y} p({y};{\theta}) \pi({\theta}\mid y^n) d{\theta}> \epsilon] = \mathrm{o}_p(n^{-2})$ for all $\epsilon > 0$, 
    There exists $t_n$ such that
    $\int_{\|{\theta} - {\theta}_0\| > \epsilon} \sup_{y} p({y};{\theta}) \pi({\theta}\mid y^n) d{\theta} = t_n/n^2$,
    where $t_n$ is a random variable that satisfies $\mathrm{P}(|t_n| > \epsilon) = \mathrm{o}(n^{-2})$
    for all $\epsilon > 0$.
    \item[A8.] For all $\epsilon > 0$, there exists $\delta > 0$ such that %$K(\boldsymbol{\theta}, \boldsymbol{\theta}_0) \geq \delta$
    \[ %D_\text{KL}(p(y; \theta) ; p(y; \theta_0)) = 
    \int p(y; \theta) \log \left(\frac{p(y; \theta)}{p(y; \theta_0)}\right) dy < \delta
 \]
    implies $ \|{\theta} - {\theta}_0 \| < \epsilon$.
\end{itemize}
}
\revise{
The assumptions A1-A5 are related to the behavior of the likelihood in the neighborhood of $\theta_0$.
The assumption A7 specifies that we only need parameter values near $\theta_0$ when calculating posterior truncated moments and the posterior density of a new observation.
For further details of the assumptions, see \cite{hartigan1998}.
}

\section*{Validation of the calibration formula in Corollary \ref{cor_calibration}}

\revise{
This section checks the validity of the calibration formula in Corollary \ref{cor_calibration}.
We employ the following $d$-variate Cauchy model with a Gaussian prior:
\begin{align*}
    Y(t) \mid \mu &\sim \text{MultivariateCauchy}(\mu,I)\quad (t=1,\ldots,n),\\
    \mu &\sim \text{Normal}(0,100 I),
\end{align*}
where the $d$-variate Cauchy density with location $\mu$ and scale matrix $I$ is given by
\[
p(y \,;\, \mu, I)
=\frac{1}{\pi^{d}}\frac{\Gamma((d+1)/2)}{\Gamma(1/2)}\frac{1}{(1+\|y-\mu\|^{2})^{(d+1)/2}}
\]
with the Gamma function $\Gamma(k)$.
The tail-heaviness slows down the convergence of MCMC algorithms especially in high dimension (c.f., \citealp{doucetal2004,Kamatani2020});
in such a case,
the calibration of the posterior mean from the MAP estimate can offer an anchor of the convergence. We employ the random-walk Metropolis--Hastings algorithms based on the multivariate Cauchy distribution with the step size determined by $0.1/\sqrt{d}$.}

\begin{figure}[h]
\label{fig S: Cauchy}
\begin{center}
\includegraphics[width=120mm]{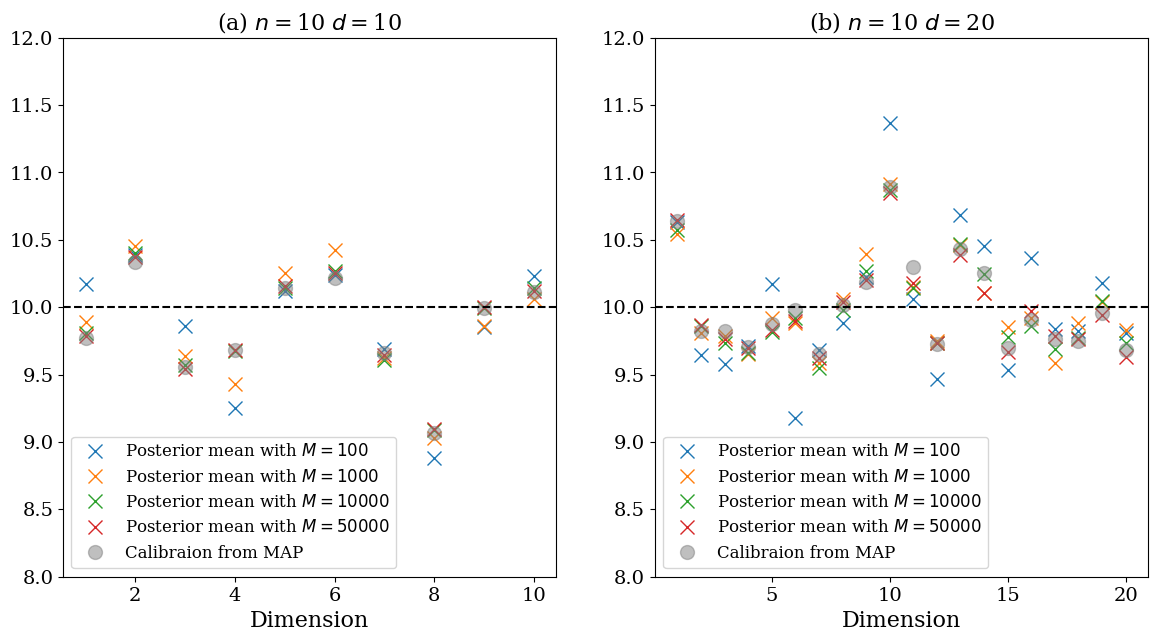}
\caption{
\revise{The calibration from the MAP estimate compared to the posterior means with the sample size $M$ of MCMC under the $d$-variate Cauchy model with the Gaussian prior. (a) the sample size $n$ is equal to 10 and the number $d$ of the dimensions is 10; (b) the sample size $n$ is equal to 10 and the number $d$ of the dimensions is 20.}
}
\end{center}
\end{figure}

\revise{
Figure 6 displays the estimate of the posterior mean using the calibration from the MAP estimate,
in comparison to posterior means based on the MCMC samples. The calibration from the MAP estimate successfully matches the posterior mean with the MCMC sample $M=50000$, and suggests that more than 10000 MCMC samples are needed for accurate computation of the posterior mean. In conclusion, this experiment demonstrates the potential application of our findings in a general setting.}

%\vspace{\fill}\pagebreak

%% ITEM 9 [See the "howt4o.tex" file.]
%\appendix
%\renewcommand{\theequation}{A\arabic{equation}}
%\setcounter{equation}{0}
%\renewcommand{\thesection}{\Alph{subsection}}
%\setcounter{section}{0}
%\section*{Appendix}
%\section*{Appendix A}
%\section*{Appendix B}
%\vspace{\fill}\pagebreak

%% ITEM 10 [See the "howto.tex" file.]
\bibliographystyle{plainnat}%           BibTeX を使う場合
\bibliography{MAPPM}% BibTeX を使う場合

\end{document}